\pdfoutput=1

\documentclass[
	bigMargin = false,
	font = palatino,
	final
	]{MyClass}

\usepackage{myMathHeader}
\let\SymTensorFieldSpace\relax
\let\MetricSpace\relax
\MyNewMathOperator{\SymTensorFieldSpace}		{command={\mathbrush{S}}}
\MyNewMathOperator{\MetricSpace}		{command={\mathbrush{M}}}

\Title{Slice theorem and orbit type stratification in infinite dimensions}

\Abstract{
	We establish a general slice theorem for the action of a locally convex Lie group on a locally convex manifold, which generalizes the classical slice theorem of Palais to infinite dimensions.
	We discuss two important settings under which the assumptions of this theorem are fulfilled.
	First, using Glöckner's inverse function theorem, we show that the linear action of a compact Lie group on a Fréchet space admits a slice.
	Second, using the Nash--Moser theorem, we establish a slice theorem for the tame action of a tame Fréchet Lie group on a tame Fréchet manifold.
	For this purpose, we develop the concept of a graded Riemannian metric, which allows the construction of a path-length metric compatible with the manifold topology and of a local addition.
	Finally, generalizing a classical result in finite dimensions, we prove that the existence of a slice implies that the decomposition of the manifold into orbit types of the group action is a stratification.
}

\Keywords{Lie group actions, slice theorem, orbit type stratification, infinite-dimensional geometry, Nash--Moser inverse function theorem}

\MSC{
	58B25,		
	(%
	58D19,  	
	58B20,		
	22E99,		
	58A35
	)}

\Institution{itp}{
	Institut für Theoretische Physik, 
	Universität Leipzig, 
	04009 Leipzig, Germany}
\Institution{mpiitp}{
	Max Planck Institute for Mathematics in the Sciences,
	04103 Leipzig, Germany
	and
	Institut für Theoretische Physik, 
	Universität Leipzig, 
	04009 Leipzig, Germany}

\Author[
	affiliation = mpiitp,
	email = tobias.diez@mis.mpg.de,
	]{diez}{Tobias Diez}
\Author[
	affiliation = itp,
	email = {rudolph@rz.uni-leipzig.de}
	]{rudolph}{Gerd Rudolph}

\begin{document}
	
\MakeTitle

\tableofcontents

\listoftodos

\section{Introduction}

In the seminal work \parencite{Palais1961}, \citeauthor{Palais1961} establishes the existence of slices for proper actions of finite-dimensional Lie groups on finite-dimensional manifolds.
This classical result shows that a proper action behaves to a large extent like the action of a compact Lie group.
In particular, the orbits are embedded submanifolds and the action on a suitable neighborhood of a given orbit can be reduced to the action of the compact stabilizer group on a slice.
As a consequence, the decomposition of the manifold into orbit type subsets is a stratification.

For many examples in physics and global analysis the symmetry groups and the manifolds are, however, infinite-dimensional.
In these cases, Palais' theorem is not applicable and the slices have to be constructed by hand, see, for example, \parencite{KondrackiRogulski1986,AbbatiCirelliEtAl1989} for gauge theory and \parencite{IsenbergMarsden1982,Ebin1970} for general relativity.
Recently, the general theory of infinite-dimensional Lie groups and their representations attracted much attention, see \eg \parencite{JanssensNeeb2015,NeebPianzol2010}.
We refer to \parencite{Neeb2006} and to the book of \textcite{GloecknerNeeb2013} (to be published) for a status report.
Nonetheless, general theorems concerning infinite-dimensional (non-linear) actions are still rare in the literature, especially, if one leaves the Banach space realm.

The aim of this paper is to present general slice theorems for the action of Fréchet Lie groups on Fréchet manifolds.
Moreover, we will show that the existence of slices implies that the decomposition of the manifold into orbit types yields a stratification.
We emphasize that the extension of these facts, well-known in the finite-dimensional setting, to infinite dimensions is not straightforward, because one faces a number of problems:
\begin{itemize}
	\item
		A closed subspace of a Fréchet or a Banach space does not need to have a topological complement.
		A simple example is provided by the closed subspace \( c_0 \subseteq l^\infty \) of sequences converging to zero in the Banach space of bounded sequences, see \parencite{Phillips1940,Whitley1966}.
		Since the translation action of \( c_0 \) on \( l^\infty \) is proper, properness does not guarantee the existence of topological complements. 
		In particular, slices need not exist for proper actions of infinite-dimensional Lie groups, even at the infinitesimal level.
	\item
		In the Banach setting, one relies on the classical inverse function theorem to extend an infinitesimal complement of the orbit to a local slice for the \( G \)-action.
		However, beyond the Banach realm, inverse functions theorems need additional assumptions such as finite dimensionality of certain spaces as in \parencite{Gloeckner2006a,Teichmann2001,Hiltunen1999} or invertibility of the derivative in a whole neighborhood as in the Nash--Moser theorem \parencite{Hamilton1982}.
	\item
		In the standard construction of a slice for a finite-dimensional Lie group action, an invariant Riemannian metric and the Riemannian exponential map play a major role.
		In infinite dimensions, Riemannian metrics are often only weakly non-degenerate and, moreover, one has to cope with issues related to the geodesic flow due to the non-existence of a general solution theory of differential equations in Fréchet spaces.
		We will comment further on these problems in \cref{cha::lcm:gradedRiemannianGeometry}.
\end{itemize}
In view of these problems, we cannot expect to recover an all-in-one slice theorem for infinite-dimensional actions.
Instead, we will pursue a modular approach.
First, we establish \cref{prop:slice:sliceTheoremGeneral} which ensures the existence of slices under very general assumptions. 
These assumptions are formulated in such a way that they can be verified in multiple ways, exploiting different features of the example one is interested in.
As an illustration of this philosophy, we will establish two further slice theorems.
First, \cref{prop:slice:sliceTheoremLinearAction} provides a slice for the linear action of a compact finite-dimensional Lie group on a Fréchet space.
Here, we use the ambient linear structure to circumvent the problems concerning the Riemannian exponential map.
Moreover, the finite-dimensionality of the group enables us to use the inverse function theorem of \textcite{Gloeckner2006a} (as we work in the context of Fréchet spaces, the more restrictive formulations of the inverse function theorem given in \parencite{Teichmann2001,Hiltunen1999} are actually sufficient for our purposes).
Second, \cref{prop::liegroup:sliceTheorem} provides a slice for the tame action of a tame Fréchet Lie group on a tame Fréchet manifold.
In this context, we use the Nash--Moser inverse function theorem phrased in the category of tame Fréchet spaces \parencite{Hamilton1982}.
As another important tool, we develop the concept of a graded Riemannian metric, which allows for constructing a path-length metric compatible with the manifold topology and a local addition.
As a special case, we recover the slice theorem of \textcite{Subramaniam1984} for elliptic actions on spaces of smooth sections of fiber bundles.
Finally, in \cref{sec:sliceTheorem:productActions}, we establish a slice theorem for Lie group product actions.

In the second part of the paper, we investigate the natural stratification of the \( G \)-manifold \( M \) and of the orbit space \( M / G \) by orbit types for a proper Lie group action of \( G \) on \( M \) admitting a slice at every point.
We establish that the existence of slices and a certain approximation property imply that the decomposition into orbit type subsets is a stratification.
In particular, we discuss local finiteness and regularity.
We show that the orbit type stratification is, in general, not locally finite by giving an example of an action of a compact group on an infinite-dimensional vector space with infinitely many orbit types, see \cref{ex:orbitTypeStratification:locallyFinite:testFunctions}.
Finally, we introduce the notion of smooth regularity of a stratification, and show that the orbit type stratification possesses this property.

The results of this paper are used in \parencite{DiezRudolphReduction,DiezThesis}, where we develop a general theory for singular symplectic reduction in infinite dimensions including an application to gauge theory.

\paragraph*{Acknowledgments}
The authors are very much indebted to M.~Schmidt for reading the manuscript and for helpful remarks.
We gratefully acknowledge support of the Max Planck Institute for Mathematics in the Sciences in Leipzig and of the University of Leipzig.

\section{Slices}

As announced above, the main results of this paper are obtained within the Fréchet category.
However, the generalities below apply to \( G \)-manifolds modeled on arbitrary locally convex spaces.
In particular, we note that we allow the model space of the manifold to vary from chart to chart.
For the differential calculus on locally convex spaces we refer to \parencite{Neeb2006,GloecknerNeeb2013}.

\subsection{Lie group actions} \label{sec::lieGroupAction} 

Let $G$ be a Lie group and let $M$ be a manifold.
A smooth map $\Upsilon: G \times M \to M$ is called a left action if 
\begin{equation}
	\Upsilon_e = \id_M \quad \text{and} \quad \Upsilon_g \circ \Upsilon_h = \Upsilon_{gh}
\end{equation}
hold for all \( g, h \in G \), where \( \Upsilon_g: M \to M \) denotes the map \( m \mapsto \Upsilon(g,m) \).
The triple \( (M, G, \Upsilon) \) will be referred to as a \emphDef{Lie group action}.
For a given Lie group \( G \), the pair \( (M, \Upsilon) \) will be called a \( G \)-manifold.
Within this terminology, further attributes apply to all elements of the triple \( (M, G, \Upsilon) \).
For example, a tame Fréchet \( G \)-manifold \( (M, \Upsilon) \) is a tame Fréchet manifold \( M \) endowed with a tame smooth action \( \Upsilon \) of a tame Fréchet Lie group \( G \).
We often omit the explicit indication of the action \( \Upsilon \).
Then, we simply write \( M \) for the \( G \)-manifold and denote the action of \( g \in G \) on \( m \in M \) by \( g \cdot m \equiv \Upsilon_g (m) \).
Furthermore, for \( m \in M \), we denote the orbit map $g \mapsto g \cdot m$ by $\Upsilon_m: G \to M$.
Its image $G \cdot m = \set{g \cdot m \given g \in G} \subseteq M$ is the orbit of the \( G \)-action through $m$.
For every element \( A \) of the Lie algebra \( \LieA{g} \) of \( G \), the Killing vector field $A^*$ on $M$ is defined by
\begin{equation}
	A^*_m \defeq \tangent_e \Upsilon_m (A), \qquad m \in M.
\end{equation}
In the dot notation, we write \( A^*_m \) as \( A \ldot m \).
If \( G \) has a smooth exponential map, then every Killing vector field \( A^* \) has a flow, namely $(m, t) \mapsto \exp(tA) \cdot m$.

The inverse image $G_m \defeq \Upsilon^{-1}_m(m) = \set{g \in G \given g \cdot m = m}$ of \( m \) under \( \Upsilon_m \) is a subgroup of \( G \), called the \emphDef{stabilizer subgroup} of $m$.
It is not known, even for Banach Lie group actions, whether \( G_m \) is always a \emph{Lie} subgroup, see \parencite[Problem~IX.3.b]{Neeb2006}.
However, in \cref{prop:properAction:stabilizerLieSubgroup} below, we will see that for proper actions this is the case.
The \( G \)-action is called free if all stabilizer subgroups are trivial.
In view of the equivariance relation \( G_{g \cdot m} = g G_m g^{-1} \), for every \( m \in M \) and \( g \in G \), we can assign to every orbit \( G \cdot m \) the conjugacy class \( (G_m) \), which is called the \emphDef{orbit type} of \( m \).
The usual partial order of subgroups transfers to orbit types by declaring\footnote{Often, the opposite order on the space of conjugacy classes is used, see \eg, \parencite[Section~2.4.14]{OrtegaRatiu2003}, \parencite[Section~4.3.1]{Pflaum2001a}. Following \parencite[Definition~2.6.1]{DuistermaatKolk1999}, the set of orbit types carries another natural order defined by saying that \( G \cdot m \) dominates \( G \cdot p \) if there exists a \( G \)-equivariant (not necessarily continuous) map \( G \cdot m \to G \cdot p \). Then, with our convention, \( (G_m) \leq (G_p) \) if and only if \( G \cdot m \) dominates \( G \cdot p \), which succinctly captures the intuition that a larger orbit has a smaller stabilizer subgroup.}
\begin{equation}\begin{split}
	(H) \leq (\tilde{H}) 	\quad 
		& \leftrightarrow \quad K \subseteq \tilde{K} \text{ for some representatives } K \in (H), \tilde{K} \in (\tilde{H}), \\
		& \leftrightarrow \quad \text{there exists } a \in G \text{ such that } a H a^{-1} \subseteq \tilde{H}.
\end{split}\end{equation} 
We say that \( H \) is \emphDef{subconjugate} to \( \tilde{H} \), if it is conjugate to a subgroup of \( \tilde{H} \).
For every closed subgroup \( H \subseteq G \), define the following subsets of \( M \):
\begin{align*}
	M_{H} &= \set{m \in M \given G_m = H} \\
	M_{\leq H} &= \set{m \in M \given G_m \subseteq H} = \bigcup_{K \subseteq H} M_{K}\\
	M_{\geq H} &= \set{m \in M \given H \subseteq G_m} = \bigcup_{K \supseteq H} M_{K}\\
\\
	M_{(H)} &= \set{m \in M \given (G_m) = (H)} \\
	M_{\leq (H)} &= \set{m \in M \given (G_m) \leq (H)} = \bigcup_{(K) \leq (H)} M_{(K)} \\
	M_{\geq (H)} &= \set{m \in M \given (H) \leq (G_m)} = \bigcup_{(K) \geq (H)} M_{(K)} \, .
\end{align*}
The subset \( M_{H} \) is called the \emphDef{isotropy type subset} and \( M_{(H)} \) is the \emphDef{subset of orbit type \( (H) \)}.
Moreover, \( M_{\geq H} \) is the set of fixed points under the action of \( H \).
Analogous definitions hold for every subset \( N \subseteq M \), so, for example, \( N_H = N \cap M_H \).

\subsection{Generalities on slices}
\label{sec::slices}

Slices provide a valuable tool to investigate group actions.
They reduce a \( G \)-action on a manifold \( M \) to an action of the stabilizer subgroup on an invariant submanifold.
In the sequel, it will be important that the stabilizer subgroup has the following property.
\begin{defn}
	A Lie subgroup \( H \subseteq G \) is called \emphDef{principal} if the natural fibration \( G \to G \slash H \) is a principal bundle.
\end{defn}
In finite dimensions, every closed subgroup of a Lie group is a principal Lie subgroup.
Moreover, every locally compact subgroup of a Banach Lie group is a principal Lie subgroup, see \parencite[Theorem~IV.3.16 and Remark~IV.4.13b]{Neeb2006}.
We refer to \parencite[Section~7.1.4]{GloecknerNeeb2013} for further details.
If the stabilizer subgroup \( G_m \) is principal, then the stabilizer subgroup of every other point of the orbit \( G \cdot m \) is principal, too.

\begin{defn}
	\label{def:slice:slice}
	Let \( M \) be a \( G \)-manifold.
	A \emphDef{slice} at \( m \in M \) is a submanifold \( S \subseteq M \) containing \( m \) with the following properties:
	\begin{thmenumerate}[label=(SL\arabic*), ref=(SL\arabic*), leftmargin=*] 
		\item \label{i::slice:SliceDefSliceInvariantUnderStab}
			The submanifold \( S \) is invariant under the induced action of the stabilizer subgroup \( G_m \), that is \( G_m \cdot S \subseteq S \).

		\item \label{i::slice:SliceDefOnlyStabNotMoveSlice}
			Any \( g \in G \) with \( (g \cdot S) \cap S \neq \emptyset \) is necessarily an element of \( G_m \). 

		\item \label{i::slice:SliceDefLocallyProduct}
			The stabilizer \( G_m \) is a principal Lie subgroup of \( G \) and the principal bundle \( G \to G \slash G_m \) admits a local section \( \chi: G \slash G_m \supseteq U \to G \) defined on an open neighborhood \( U \) of the identity coset \( \equivClass{e} \) in such a way that the map
			\begin{equation}
				\chi^S: U \times S \to M, \qquad (\equivClass{g}, s) \mapsto \chi(\equivClass{g}) \cdot s
			\end{equation}
			is a diffeomorphism onto an open neighborhood \( V \subseteq M \) of \( m \). We call \( V \) a \emphDef{slice neighborhood} of \( m \).

		\item \label{i::slice:SliceDefPartialSliceSubmanifold}
			The partial slice \( S_{(G_m)} = \set{s \in S \given G_s \text{ is conjugate to } G_m} \) is a closed submanifold of \( S \).
			\qedhere
	\end{thmenumerate}
\end{defn}

\begin{remark}
	In the infinite-dimensional context, many variations of the concept of a slice are discussed in the literature.
	Most of these notions of a slice are weaker than our definition.
	For example, \parencite[Section~5]{FischerMarsdenEtAl1980} and \parencite[Definition~44.17]{KrieglMichor1997} require the map \( \chi^S \) in \iref{i::slice:SliceDefLocallyProduct} to be merely a homeomorphism instead of a diffeomorphism.
	In \parencite[Definition~2.3.1]{DuistermaatKolk1999} and \parencite[Definition~4.2.2]{Sniatycki2013} the condition \iref{i::slice:SliceDefLocallyProduct} is replaced by its infinitesimal version:
	\begin{align}
		\TBundle_m M &= \TBundle_m (G \cdot s) \oplus \TBundle_m S, \\
		\TBundle_s M &= \TBundle_s (G \cdot s) + \TBundle_s S \quad \text{ for all } s \in S. 
	\end{align}
	These transversality relations imply \iref{i::slice:SliceDefLocallyProduct} if the inverse function theorem is available, see \parencite[Proof of Theorem~2.3.26]{OrtegaRatiu2003}.
	None of these other definitions include \iref{i::slice:SliceDefPartialSliceSubmanifold}, which is, however, crucial in the study of orbit type subsets.
	We will see in \cref{prop:properAction:sliceOrbitTypeEqualsSliceStab} that \iref{i::slice:SliceDefPartialSliceSubmanifold} holds in certain situations automatically.
	
	Yet another generalization of the notion of a slice, applicable to group actions on stratified spaces, is proposed in \parencite{IsenbergMarsden1982}. 
\end{remark}
We now record a few important consequences resulting from the existence of slices. 
\begin{prop}
	Let \( M \) be a \( G \)-manifold.
	Assume that there exists a slice \( S \) at the point \( m \in M \).
	Then, the following holds: 
	\begin{thmenumerate}
		\item \label{prop::slice:SliceDefSliceSweepIsOpen} 
			The \( G \)-closure \( G \cdot S \) of \( S \) is an open neighborhood of the orbit \( G \cdot m \) in \( M \).
			Moreover, \( S \) is closed in \( G \cdot S \). 
		\item \label{prop::slice:mHasMaximalStabilizerOfWholeSlice} 
			For every \( s \in S \), the stabilizer \( G_s \) is a subgroup of \( G_m \).
			In other words, the point \( m \) has the maximal stabilizer of the whole slice.
			Symbolically, \( S \subseteq M_{\leq G_m} \).
		
		\item \label{prop::slice:sliceNeighbourhoodHasSubconjugatedStabilizer}
			There exists an open neighborhood \( V \subseteq M \) of \( m \) such that any point of \( V \) has a stabilizer that is conjugate to a subgroup of \( G_m \).
			In other words, \( V \subseteq M_{\leq (G_m)} \). \qedhere
	\end{thmenumerate}
\end{prop}
\begin{proof}
	The set \( G \cdot S \) is open, because it is the union of the open sets \( \bigUnion_{g \in G} g \cdot V \), where \( V \) is the open slice neighborhood of \iref{i::slice:SliceDefLocallyProduct}.
	Furthermore, \( S \) is closed in \( G \cdot S \), because \( \set{\equivClass{e}} \times S \) is closed in \( U \times S \) and thus \( S \) is closed in the slice neighborhood \( V \) by \iref{i::slice:SliceDefLocallyProduct}.
	The second claim follows directly from property \iref{i::slice:SliceDefOnlyStabNotMoveSlice}.
	For the third point, let \( V \) be the slice neighborhood as described in \iref{i::slice:SliceDefLocallyProduct}.
	Since every \( v \in V \) is obtained by translation of a point in the slice \( S \), there exists \( g \in G \) such that \( g \cdot v \in S \).
	Using equivariance of stabilizer subgroups, we conclude \( g G_v g^{-1} = G_{g \cdot v} \subseteq G_m \).
\end{proof}
\Cref{prop::slice:sliceNeighbourhoodHasSubconjugatedStabilizer} is called the neighboring subgroups theorem.
In the finite-dimensional context, it is due to \textcite{MontgomeryZippin1942}.
Clearly, an action violating the neighboring subgroups property cannot have slices.
This observation was used in \parencite{MichorSchichl1998} to show that, for a fiber bundle \( F \), the natural action of the group of vertical automorphism of \( F \) on the space of (generalized) connections on \( F \) does not admit a slice.

We now discuss the related concepts of an equivariant retraction and a tubular neighborhood.
\begin{defn}
	Let \( M \) be a \( G \)-manifold and let \( O \subseteq M \) be a \( G \)-orbit.
	We say that the \( G \)-action has a \emphDef{tubular neighborhood} of \( O \) if for some \( m \in O \) the following holds: 
	\begin{thmenumerate}
		\item
			The stabilizer \( G_m \) is a principal Lie subgroup of \( G \).
		\item
			There exists a \( G_m \)-invariant submanifold \( S \) containing \( m \) and a \( G \)-equivariant diffeomorphism \( \chi^{\tube}: G \times_{G_m} S \to W \) onto an open, \( G \)-invariant neighborhood \( W \) of \( G \cdot m \) in \( M \) such that \( \chi^{\tube}(\equivClass{e, s}) = s \) for all \( s \in S \).
		\item
			The subset \( S_{(G_m)} \) is a closed submanifold of \( S \).
	\end{thmenumerate}
	The diffeomorphism \( \chi^{\tube} \) is called the tube diffeomorphism.
\end{defn}
If these conditions hold for some \( m \in O \), then they clearly hold for every other point of \( O \), too.

\begin{prop}
	Let \( M \) be a \( G \)-manifold. 
	Assume that there exists a slice \( S \) at the point \( m \in M \).
	Then, the following hold:
	\begin{thmenumerate}
		\item 
			There exists a smooth \( G \)-equivariant retraction\footnote{A retraction of the inclusion \( A \toInject X \) of a subspace \( A \) in a topological space \( X \) is a continuous map \( r: X \to A \) such that the restriction of \( r \) to \( A \) is the identity on \( A \), that is, \( r(a) = a \) for \( a \in A \).} \( r: G \cdot S \to G \cdot m \) of the injection \( G \cdot m \toInject G \cdot S \) such that \( S = r^{-1}(m) \).
		\item
			\label{prop:slice:existenceTube}
			The twisted product \( G \times_{G_m} S \) yields a tubular neighborhood of the orbit \( G \cdot m \).
			That is, there exist a \( G \)-equivariant diffeomorphism \( \chi^\tube: G \times_{G_m} S \to M \) onto the open \( G \)-invariant neighborhood \( G \cdot S \) of \( G \cdot m \) in \( M \).
			\qedhere
	\end{thmenumerate}
\end{prop}
\begin{proof}
	Since \( S \) is a slice at \( m \in M \), the set \( G \cdot S \) is open in \( M \) according to \cref{prop::slice:SliceDefSliceSweepIsOpen}.
	Moreover, the following map is well-defined:
	\begin{equation}
		r: G \cdot S \to G \cdot m, \qquad g \cdot s \mapsto g \cdot m.
	\end{equation}
	Indeed, if \( g' \cdot s' = g \cdot s \) is another representation, then \( g^{-1} g' \) is necessarily an element of the stabilizer \( G_m \) by \iref{i::slice:SliceDefOnlyStabNotMoveSlice}.
	Hence, $r(g \cdot s) = g \cdot m = g' \cdot m = r(g' \cdot s')$.
	It is clear that \( r \) is a \( G \)-equivariant retraction.
	Due to the invariance of the slice under the stabilizer action \iref{i::slice:SliceDefSliceInvariantUnderStab}, we have \( r^{-1}(m) = G_m \cdot S = S \).
	Hence, it remains to show that \( r \) is a smooth map.
	By \( G \)-equivariance it is sufficient to inspect smoothness near the slice.
	Thus, we can restrict attention to the slice neighborhood \( V \isomorph U \times S \) of \iref{i::slice:SliceDefLocallyProduct}.
	Now, it is not hard to see that in these product coordinates \( r \) corresponds to the projection on the \( U \)-factor and hence it is smooth.

	We next show how to construct a tubular neighborhood starting from a slice.
	By definition, \( S \) is invariant under the stabilizer and, hence, the twisted product \( G \times_{G_m} S \) is well-defined.
	The set \( G \times_{G_m} S \) carries a natural smooth manifold structure derived from the local trivializations of \( G \to G \slash G_m \).
	The desired tube diffeomorphism is given by
	\begin{equation}
		\chi^\tube: G \times_{G_m} S \to G \cdot S, \qquad \equivClass{g, s} \mapsto g \cdot s.
	\end{equation}
	Indeed, \( G \cdot S \) is an open \( G \)-invariant neighborhood of \( G \cdot m \) by \cref{prop::slice:SliceDefSliceSweepIsOpen} and \( \chi^\tube \) is obviously surjective.
	Injectivity is also easily verified, as \( g' \cdot s' = g \cdot s \) directly implies that \( g^{-1} g' \) is an element of the stabilizer \( G_m \) due to \iref{i::slice:SliceDefOnlyStabNotMoveSlice}.
	To complete the proof we are left with showing that \( \chi^\tube \) is a local diffeomorphism.
	By \( G \)-equivariance and \iref{i::slice:SliceDefLocallyProduct}, we can again restrict attention to an open neighborhood \( V \) of \( m \) which is diffeomorphic to \( U \times S \), where \( U \subseteq G/G_m \) is an open neighborhood of the identity coset and the diffeomorphism is induced by a local section \( \chi: G/G_m \supseteq U \to G \).
	In these coordinates and relative to the local trivialization of \( G \times_{G_m} S \) induced by the local section \( \chi \), the map \( \chi^\tube \) is identified with the diffeomorphism \( \chi^S: U \times S \to V \).
\end{proof}
Conversely, we have the following.
\begin{prop}
	Let \( M \) be a \( G \)-manifold.
	If the \( G \)-action has a tubular neigh\-borhood of the \( G \)-orbit \( O \), then there exists a slice at every point of \( O \).
\end{prop}
\begin{proof}
	Let \( m \in O \) and let \( \chi^{\tube}: G \times_{G_m} S \to W \) be a tube diffeomorphism.
	We claim that \( S \) is a slice.
	The properties \iref{i::slice:SliceDefSliceInvariantUnderStab} and \iref{i::slice:SliceDefPartialSliceSubmanifold} of a slice hold by assumption.
	In order to verify \iref{i::slice:SliceDefOnlyStabNotMoveSlice}, let \( g \in G \) and \( s \in S \) be such that \( g \cdot s \in S \).
	Since \( \chi^{\tube} \) is \( G \)-equivariant and satisfies \( \chi^{\tube}(\equivClass{e, \tilde{s}}) = \tilde{s} \) for all \( \tilde{s} \in S \), we have
	\begin{equation}
		\chi^{\tube}(\equivClass{g, s}) = g \cdot s = \chi^{\tube}(\equivClass{e, g \cdot s}).
	\end{equation}
	Hence, \( \equivClass{g, s} = \equivClass{e, g \cdot s} \).
	The latter is only possible if \( g \in G_m \); and so \iref{i::slice:SliceDefOnlyStabNotMoveSlice} holds.
	Since \( G_m \) is a principal Lie subgroup of \( G \), the \( G_m \)-bundle \( G \to G \slash G_m \) has a local section \( \chi: G \slash G_m \supset U \to G \) defined in an open neighborhood \( U \) of the identity coset.
	Let \( \tau: U \times S \to \restr{(G \times_{G_m} S)}{U} \) be the induced local trivialization of \( G \times_{G_m} S \) defined by
	\begin{equation}
		\tau(\equivClass{g}, s) \defeq \equivClass{\chi(\equivClass{g}), s}.
	\end{equation}
	By \( G \)-equivariance of \( \chi^{\tube} \), we obtain
	\begin{equation}
		\chi^{\tube} \circ \tau (\equivClass{g}, s) 
			= \chi^\tube(\equivClass{\chi(\equivClass{g}), s})
			= \chi(\equivClass{g}) \cdot \chi^\tube(\equivClass{e, s}) 
			= \chi(\equivClass{g}) \cdot s.
	\end{equation}
	Since \( \chi^\tube \) is a diffeomorphism onto its image the map \( \chi^S \defeq \chi^{\tube} \circ \tau \) is a diffeomorphism onto an open neighborhood of \( m \) in \( M \).
	This verifies \iref{i::slice:SliceDefLocallyProduct}.
\end{proof}
In the finite-dimensional setting, one can construct a slice from a smooth retraction, see \parencite[Theorem~2.3.26]{OrtegaRatiu2003}.
However, to verify the property \iref{i::slice:SliceDefLocallyProduct} the inverse function theorem is needed and thus the construction does not carry over to arbitrary locally convex manifolds.
Hence, in our infinite-dimensional setting, a slice (or a tubular neighborhood) is a stronger concept than a retraction.

A slice \( S \) reduces the description of the action on some neighborhood to the action of the stabilizer group on \( S \).
It is often convenient to further simplify the situation and reduce the investigation to a \emph{linear} action.
\begin{defn}
	\label{def:slice:linearSlice}
	Let \( M \) be a \( G \)-manifold.
	A slice \( S \) at a point \( m \in M \) is called a \emphDef{linear slice} if the following holds:
	\begin{enumerate}[label=(SL\arabic*), ref=(SL\arabic*), start=5, leftmargin=*]
		\item 
			\label{i:slice:linearSlice}
			There exist a continuous representation of \( G_m \) on a locally convex vector space \( X \) and a \( G_m \)-equivariant diffeomorphism \( \iota_S \) from a \( G_m \)-invariant open neighborhood of \( 0 \) in \( X \) onto \( S \) such that \( \iota_S(0) = m \).
			\qedhere
	\end{enumerate}
\end{defn}
Note that \( \tangent_0 \iota_S: X \to \TBundle_m S \) identifies \( X \) with the tangent space to the slice at \( m \).
We usually suppress the diffeomorphism \( \iota_S \) and view \( S \) as a submanifold of \( M \) or as an open subset of \( X \) depending on the context.

In finite dimensions, a classical result by \textcite[Theorem~1]{Bochner1945} shows that every action of a compact Lie group can be linearized near a fixed point.
Hence, every slice of a proper action is linear (after shrinking it).
However, the proof of Bochner's linearization theorem relies on the classical inverse function theorem and so it does not directly generalize to the infinite-dimensional setting.
Nonetheless, the forthcoming slice \cref{prop:slice:sliceTheoremGeneral} shows that the resulting slice is linear, indeed.

\subsection{Proper actions}
General group actions may exhibit pathological and unwanted features.
For example, the orbit space might be non-Hausdorff.
In the finite-dimensional setting, it is well known that proper actions behave significantly better and, at the same time, the class of proper actions is wide enough to include interesting examples.
These pleasant properties continue to hold in the infinite-dimensional context, as we will discuss now.

Recall that a continuous map is called proper if the inverse image of every compact subset is compact.
A group action \( \Upsilon: G \times M \to M \) is said to be \emphDef{proper} if the map
\begin{equation}
	\Upsilon_\ext: G \times M \to M \times M, \qquad (g,m) \mapsto (g \cdot m, m) 	
\end{equation}
is a proper map.
Checking properness directly, based on this definition, can be intricate. 
The following gives more practical characterizations.
\begin{prop}
	\label{prop:lieGroupAction:properEquivalent}
	Let \( (M, G, \Upsilon) \) be a left Lie group action.
	The statements (i), (ii) and (iii) below are equivalent to properness of $\Upsilon$ for adequate conditions on the spaces involved:
	\\[1.5ex]
	\noindent
	If $M$ is a compactly generated Hausdorff space:
	\begin{thmenumerate}[series=properEquivalent]
		\item
			\label{prop:lieGroupAction:proper:closedOribtMapCompactStabilizer}
			$\Upsilon_\ext$ is a closed map and every stabilizer subgroup is compact. 
	\end{thmenumerate}
	If $G$ and $M$ are metric spaces:
	\begin{thmenumerate}[resume=properEquivalent]
		\item
			\label{prop:lieGroupAction:proper:sequenceMetric}
			Let $(g_i)$ and $(m_i)$ be sequences in $G$ and $M$, respectively.
			If $(m_i)$ and $(g_i \cdot m_i)$ converge, then the sequence $(g_i)$ has a convergent subsequence.
	\end{thmenumerate}
	If $G$ has a complete left-invariant metric $d_G$ compatible with the Lie group topology and $M$ admits a $G$-invariant metric $d_M$ compatible with the manifold topology:
	\begin{thmenumerate}[resume=properEquivalent]
		\item
			\label{prop:lieGroupAction:proper:sequenceMetricInvariant}
			Let \( m \in M \).
			Every sequence $(g_i)$ in $G$ for which $(g_i \cdot m)$ converges to $m$ has a convergent subsequence.\footnote{In this case, the subsequence necessarily converges to an element of the stabilizer $G_m$.}
			\qedhere
	\end{thmenumerate}	
\end{prop}
\begin{proof}
	The proof of the first two characterizations is a routine exercise in topology and the proof of the third one can be found in \parencite[p.~42]{Subramaniam1984}.	
\end{proof}
As one would expect, properness of the group action has pleasant consequences for the topological properties of orbits and orbit spaces.
\begin{prop}\label{prop:lieGroupProperAction:properties}
	Let \( (M, G, \Upsilon) \) be a proper Lie group action.
	Assume that the topology of \( M \) is compactly generated.
	Then, the following holds for every point \( m \in M \):
	\begin{thmenumerate}
		\item
			The orbit map \( \Upsilon_m: G \to M \) is proper and the orbit \( G \cdot m \) is closed in \( M \).
			Moreover, \( \Upsilon_m \) descends to a map \( \check\Upsilon_m: G \slash G_m \to M \), which is a homeomorphism onto \( G \cdot m \).
		\item \label{prop:properAction:stabilizerCompact}
			The stabilizer subgroup \( G_m \) is compact.
		\item \label{prop:properAction:orbitSpaceHausdorff}
			\( G \slash G_m \) and \( M / G \), both endowed with the quotient topology, are Hausdorff.
			\qedhere
	\end{thmenumerate}
\end{prop}
\begin{proof}
	\begin{thmenumerate}*
		\item
			Properness of \( \Upsilon_m \) is a direct consequence of the relation \( \Upsilon_\ext(\cdot, m) = \Upsilon_m(\cdot) \times \set{m} \).
			Clearly, \( \Upsilon_m \) descends to an injective map \( \check\Upsilon_m: G \slash G_m \to M \).
			Being proper, \( \Upsilon_m \) is a closed map and hence the image \( \Upsilon_m(G) = G \cdot m \) is a closed subset of \( M \). 
			The map \( \check\Upsilon_m \) is a homeomorphism onto its image, because it is a continuous and closed injection. 
		\item
			As the preimage of \( \set{m} \) under the proper map \( \Upsilon_m: G \to M \), the stabilizer \( G_m \) is compact.
		\item
			For the Hausdorff property of the quotient spaces, recall that the codomain \( Y \) of a surjective, continuous, open map \( f:X \to Y \) is Hausdorff if and only if \( R_f = \set{(x_1, x_2) \in X \times X \given f(x_1) = f(x_2)} \) is closed in \( X \times X \), \eg \parencite[Lemma 3.2]{Dieck1987}.
			In the present case, the maps 
			\begin{equation}
				\pi_{G_m}: G \to G/G_m \quad \text{and} \quad \pi: M \to M / G
			\end{equation}
			are quotient maps with respect to continuous group actions and so they are surjective, continuous and open.
			Furthermore, 
			\begin{align}
				R_{\pi_{G_m}} &= \set{(g,h) \in G \times G \given g G_m = h G_m} \\
				\intertext{is closed as the inverse image of the closed subset \( G_m \) under the continuous map \( G \times G \ni (g,h) \mapsto g^{-1}h \in G \). Similarly,}
				R_{\pi} &= \set{(m, \tilde{m}) \in M \times M \given G \cdot m = G \cdot \tilde{m}}
			\end{align}
			is closed, because it is the image \( \Upsilon_\ext(G,M) \) of the closed map \( \Upsilon_\ext \).
			\qedhere
	\end{thmenumerate}
\end{proof}

It is still an open question, even for actions of Banach Lie groups, whether stabilizers are always Lie subgroups, see \parencite[Problem IX.3.b]{Neeb2006}.
The situation is better for proper actions of locally exponential\footnote{A Lie group is called locally exponential if the exponential map exists and is a local diffeomorphism.} Lie groups as the following result shows.
\begin{lemma}
	\label{prop:properAction:stabilizerLieSubgroup}
	Let \( G \) be a locally exponential Lie group that acts properly on a manifold \( M \).
	Then, every stabilizer subgroup is a finite-dimensional principal Lie subgroup of \( G \). 
\end{lemma}
\begin{proof}
	Let \( m \in M \).
	By \cref{prop:properAction:stabilizerCompact}, the stabilizer \( G_m \) is compact.
	Thus, \parencite[Theorem~7.3.14]{GloecknerNeeb2013} shows that \( G_m \) is a finite-dimensional Lie subgroup of \( G \).
	Due to the finite dimensionality of \( G_m \), by \parencite[Theorem~G]{Gloeckner2015}, the quotient \( G \slash G_m \) has a smooth manifold structure such that the canonical projection \( G \to G \slash G_m \) is a submersion.
	Hence, \( G_m \) is a principal Lie subgroup.
\end{proof}
For later purposes, let us record a helpful property of compact subgroups.
\begin{lemma} \label{prop::compactLieSubgroup:conjugatedSubgroupEqual}
	Let \( G \) be a Lie group. 
	Let \( H \) and \( K \) be two compact Lie subgroups of \( G \). 
	If \( K \) is conjugate to \( H \) and \( K \subseteq H \), then \( K = H \).
\end{lemma}
\begin{proof}
	The proof for finite-dimensional \( G \) carries over word by word to infinite dimensions, compare \parencite[Lemma~4.2.9]{Pflaum2001a} or \parencite[Lemma~2.1.14]{OrtegaRatiu2003}.
	Indeed, being compact, the Lie subgroups \( H \) and \( K \) are finite-dimensional.
	Moreover, by assumption, there exists \( g \in G \) such that \( g K g^{-1} = H \).
	Since conjugation by \( g \) is a diffeomorphism of \( G \), the dimensions of \( H \) and \( K \) have to be equal.
	Thus, \( K \) is open and closed in \( H \) and, therefore, it consists of connected components of \( H \).
	By compactness, both groups only have finitely-many connected components.
	Hence it is enough to show that the number of connected components coincide.
	But this is, of course, a direct consequence of the relation \( g K g^{-1} = H \).
\end{proof}

Properness of the action has pleasant consequences for the behavior of the orbit type subsets, too. 
\begin{prop}
	Let \( G \) act properly on \( M \).
	Assume that every stabilizer subgroup is a Lie subgroup of \( G \).
	Then, the following statements hold:
	\begin{thmenumerate}
		\item \label{prop::properAction:orbitTypeIsPartialOrder}
			The preorder \( \leq \) of orbit types is in fact a partial order\footnotemark.
			\footnotetext{A preorder is a reflexive and transitive binary relation. An antisymmetric preorder is called a partial order.}
		\item \label{prop:properAction:relationBetweenStabSubsets}
			For all stabilizer subgroups \( H \subseteq G \), we have:
			\begin{subequations}\label{eq:properAction:relationBetweenStabSubsets}\begin{align}
				M_{\leq (H)} \intersect M_{\geq (H)} &= M_{(H)} \, , \\
				M_{\geq H} \intersect M_{\leq (H)} &= M_H = M_{\geq H} \intersect M_{(H)} \, , \\
				M_{\leq H} \intersect M_{\geq (H)} &= M_H = M_{\leq H} \intersect M_{(H)} \, .
			\end{align}\end{subequations}
		\item \label{prop:properAction:sliceOrbitTypeOpenInSupOrbitType}
			If, in addition, the action admits a slice at every point, then \( M_{(H)} \) is open in \( M_{\geq (H)} \)  and \( M_H \) is open in \( M_{\geq H} \).
		\item \label{prop:properAction:sliceOrbitTypeEqualsSliceStab}
			We have \( S_{(G_m)} = S_{G_m} \) for every slice \( S \) at \( m \in M \).
			In particular, for a linear slice \( S \), the property~\iref{i::slice:SliceDefPartialSliceSubmanifold} holds automatically.
			\qedhere
	\end{thmenumerate}
\end{prop}
\begin{proof}
	\begin{thmenumerate}*
		\item
			Antisymmetry of \( \leq \) is a direct consequence of \cref{prop::compactLieSubgroup:conjugatedSubgroupEqual}.
			Indeed, all stabilizer subgroups are compact, because the action is proper.
			If \( H \) and \( K \) are stabilizer subgroups with \( (K) \leq (H) \) and \( (H) \leq (K) \), then there exist \( a, b \in G \) such that \( a K a^{-1} \subseteq H \) and \( b H b^{-1} \subseteq K \).
			Combining both relations yields \( (ab) H (ab)^{-1} \subseteq a K a^{-1} \subseteq H \).
			Now, \cref{prop::compactLieSubgroup:conjugatedSubgroupEqual} implies \( (ab) H (ab)^{-1} = H \).
			Therefore, \( a K a^{-1} = H \) and so \( (K) = (H) \).
		\item
			The identities~\eqref{eq:properAction:relationBetweenStabSubsets} follow from \labeliref{prop::properAction:orbitTypeIsPartialOrder} or directly from \cref{prop::compactLieSubgroup:conjugatedSubgroupEqual}.
			For example, if \( m \in M_{\geq H} \intersect M_{\leq (H)} \), then \( G_m \supseteq H \) and there exists \( a \in G \) such that \( a G_m a^{-1} \subseteq H \).
			Therefore, \( a G_m a^{-1} \subseteq H \subseteq G_m \) and thus \( G_m = H \) by \cref{prop::compactLieSubgroup:conjugatedSubgroupEqual}.
		\item
			Let \( m \in M_{(H)} \) and let \( S \) be a slice at \( m \).
			The slice neighborhood \( V \) provided by \iref{i::slice:SliceDefLocallyProduct} is open in \( M \) and hence the intersection \( V \intersect M_{\geq (H)} \) is open in \( M_{\geq (H)} \).
			However, \( V \subseteq M_{\leq (G_m)} = M_{\leq (H)} \) by \cref{prop::slice:sliceNeighbourhoodHasSubconjugatedStabilizer} and thus we have found an open neighborhood of \( m \) in \( M_{\geq (H)} \) that is contained in \( M_{\geq (H)} \intersect M_{\leq (H)} = M_{(H)} \). 

			Openness of \( M_H \) in \( M_{\geq H} \) follows from \labeliref{prop:properAction:relationBetweenStabSubsets} by a similar argument using the relation \( M_{\geq H} \intersect M_{\leq (H)} = M_H \).
		\item
			Let \( S \) be a slice at \( m \in M \).
			The point \( m \) has the maximal stabilizer subgroup of the whole slice, see \cref{prop::slice:mHasMaximalStabilizerOfWholeSlice}.
			Thus, using \labeliref{prop:properAction:relationBetweenStabSubsets}, we get \( S_{(G_m)} \subseteq M_{\leq G_m} \intersect M_{(G_m)} = M_{G_m} \) and hence \( S_{(G_m)} \subseteq S_{G_m} \).
			The converse inclusion is trivial.

			Finally, suppose that \( S \) is a linear slice, which is \( G_m \)-equivariantly diffeomorphic to an open subset of the locally convex vector space \( X \).
			Note that \( X_{G_m} \) is a closed subspace of \( X \).
			Thus, \( S_{(G_m)} = S_{G_m} \) implies that the partial slice \( S_{(G_m)} \) is a closed submanifold of \( S \). 
			\qedhere
	\end{thmenumerate}
\end{proof}

\section{Slice Theorem}

\subsection{General slice theorem}

The main idea for the construction of the slice at a point \( m \in M \) is to move from \( m \) in the direction normal to the orbit.
This does not work per se on a naked manifold, but needs the following additional structure.
\begin{defn}[\textnormal{\cf \parencite[Section~42.4]{KrieglMichor1997}}]
	\label{def:slice:localAddition}
	A \emphDef{local addition} is a smooth map \( \eta: \TBundle M \supseteq U \to M \) defined on an open neighborhood \( U \) of the zero section in \( \TBundle M \) such that
	\begin{thmenumerate}
		\item
			the composition of \( \eta \) with the zero section of \( \TBundle M \) is the identity on \( M \), \ie, \( \eta(m, 0) = m \) for all \( m \in M \),
		\item
			the map \( \pr \times \eta: \TBundle M \supseteq U \to M \times M \) is a diffeomorphism onto an open neighborhood of the diagonal, where \( \pr: \TBundle M \to M \) is the canonical projection.
	\end{thmenumerate}
	For \( M \) being a \( G \)-manifold, a local addition \( \eta \) is called \emphDef{\( G \)-equivariant} if \( U \) is an \( G \)-invariant subset of \( \TBundle M \) and \( \eta: U \to M \) is \( G \)-equivariant.  
\end{defn}

\begin{example}[affine local addition]
	\label{ex:slice:localAddition:linear}
	Let \( X \) be a locally convex vector space.
	Then, the map
	\begin{equation}
		\eta: \TBundle X \isomorph X \times X \to X, \qquad (x, v) \mapsto x + v
	\end{equation}
	is a local addition.
	If \( X \) is endowed with a linear action of a Lie group \( G \), then \( \eta \) is \( G \)-equivariant. 
	This example directly generalizes to affine spaces and open subsets thereof.
\end{example}

\begin{example}[local additions on Lie groups]
	\label{ex:slice:localAddition:group}
	Let \( G \) be a Lie group.
	For a chart \( \kappa: G \supseteq V \to V' \subseteq \LieA{g} \) at the identity, put \( U \defeq \bigDisjUnion_{g \in G} g \ldot V' \subseteq \TBundle G \).
	Then, the map
	\begin{equation}
		\eta: \TBundle G \supseteq U \to G, \qquad g \ldot \xi \mapsto g \, \kappa^{-1}(\xi)
	\end{equation}
	is a left-equivariant local addition on \( G \).
\end{example}
As we will see in \cref{cha::lcm:gradedRiemannianGeometry}, the exponential map of a Riemannian metric yields a local addition, too.

\begin{defn}
	\label{defn:slice:localAddition:adapted}
	We say that a local addition \( \eta: \TBundle M \supseteq U \to M \) is \emphDef{adapted to a submanifold \( P \subseteq M \)} with an embedded normal bundle \( \NBundle P \subseteq \restr{\TBundle M}{P} \) if the restriction of \( \eta \) to \( U \intersect \NBundle P \) is a local diffeomorphism at every point of the zero section of \( \NBundle P \).
\end{defn}
With a slight abuse of notation, the restriction of \( \eta \) to \( U \intersect \NBundle P \) will be denoted by \( \eta \) as well. 

\begin{thm}[General slice theorem]
	\label{prop:slice:sliceTheoremGeneral}
	 Let \( M \) be a \( G \)-manifold with a proper \( G \)-action and let \( m \in M \).
 	Assume that the following conditions hold:
 	\begin{enumerate}
 		\item
 			\label{i:slice:sliceTheoremGeneral:stabilizer}
 			The stabilizer \( G_m \) is a principal Lie subgroup of \( G \).
 		\item
 			\label{i:slice:sliceTheoremGeneral:orbit}
 			The orbit \( O \defeq G \cdot m \) is a locally closed submanifold whose normal bundle \( \NBundle O \) is realized as a smooth \( G \)-invariant subbundle of \( \restr{\TBundle M}{O} \).
 		\item 
 			\label{i:slice:sliceTheoremGeneral:localAddition}
 			There exists a \( G \)-equivariant local addition \( \eta: \TBundle M \supseteq U \to M \) adapted to the submanifold \( O \).
 		\item
 			\label{i:slice:sliceTheoremGeneral:metric}
 			\( M \) carries a \( G \)-invariant topological metric \( d \) compatible with the manifold topology on \( M \) and \( \TBundle M \) is endowed with a \( G \)-invariant topological fiber metric \( \rho \) compatible with the topology on the fibers of \( \TBundle M \).
 	\end{enumerate}
 	Then, the \( G \)-action on \( M \) admits a linear slice at \( m \).
\end{thm}
Informally, the conditions in \cref{prop:slice:sliceTheoremGeneral} are used for the construction of the slice in the following way.
Assumption~\iref{i:slice:sliceTheoremGeneral:stabilizer} ensures that \( G \slash G_m \) is a smooth manifold.
According to~\iref{i:slice:sliceTheoremGeneral:orbit}, it is diffeomorphic to the orbit.
Moreover, condition~\iref{i:slice:sliceTheoremGeneral:orbit} also ensures that, on the infinitesimal level, there is a complement to the orbit in \( \TBundle_m M \).
Using~\iref{i:slice:sliceTheoremGeneral:localAddition}, we are able to move along the normal direction to obtain a submanifold \( S \) transversal to the orbit.
Finally, using the metrics given in~\iref{i:slice:sliceTheoremGeneral:metric}, we can control the size of \( S \) to make sure that it is a slice, indeed. 

\begin{remark}
	In the setting of \cref{prop:slice:sliceTheoremGeneral}, if we additionally assume that \( \TBundle_m M \) be complete for every \( m \in M \), then the assumption about the existence of \( \rho \) in condition~\iref{i:slice:sliceTheoremGeneral:metric} implies that \( M \) is a Fréchet manifold.
\end{remark}

The proof of \cref{prop:slice:sliceTheoremGeneral} will follow from \cref{prop:slice:sliceTheoremGeneral:diffeoOrbit,prop:slice:sliceTheoremGeneral:sliceConstruction} below.
The slice will be constructed as the image of \( \NBundle_m O \) under the local addition \( \eta \).
In order for \iref{i::slice:SliceDefLocallyProduct} to hold, the \emph{local} diffeomorphism $\eta: U \intersect \NBundle O \to M$ needs to be enhanced to a diffeomorphism of a neighborhood of the zero-section of \( \NBundle O \) onto an open neighborhood of the orbit in \( M \).
To move from the local picture to a semi-global description along the orbit, we will use a simple metric argument following \parencites[Proposition~2.3]{Ramras2011}[p.~64f]{Subramaniam1984}.

\begin{lemma}
	\label{prop:slice:sliceTheoremGeneral:diffeoOrbit}
	In the setting of \cref{prop:slice:sliceTheoremGeneral}, there exists a \( G \)-invariant open neighborhood \( V \) of the zero section in \( \NBundle O \) with \( V \subseteq U \intersect \NBundle O \) such that the restriction of \( \eta \) to \( V \) is a \( G \)-equivariant diffeomorphism onto an open neighborhood of \( O \) in \( M \).
\end{lemma}
\begin{proof}
	Since \( \eta: U \intersect \NBundle O \to M \) is a local diffeomorphism at every point of the zero section over \( O \), it is enough to show injectivity of \( \eta \) on some open neighborhood of the zero section in $\NBundle O$.
	First, note that the ball
	\begin{equation}
		\OpenBall_\varepsilon(m,0) \defeq \set{X_p \in \TBundle M \given d(m, p) < \varepsilon, \rho_{p}(X_p, 0) < \varepsilon}
	\end{equation}
	is an open neighborhood of \( (m, 0) \) in \( \TBundle M \).
	Since the local addition \( \eta \) restricts to a local diffeomorphism \( U \intersect \NBundle O \to M \) and since \( \NBundle O \) is a subbundle of \( \TBundle M \), we can assume that \( \eta \) is a diffeomorphism from the open neighborhood \( U_\varepsilon(m, 0) \defeq \OpenBall_\varepsilon(m, 0) \intersect \NBundle O \) of \( (m, 0) \) in \( \NBundle O \) onto an open neighborhood \( W_\varepsilon(m) \) of \( m \) in \( M \) for some sufficiently small \( \varepsilon \).
	Since \( W_\frac{\varepsilon}{2}(m) \) is an open neighborhood of \( m \), there exists \( \delta < \frac{\varepsilon}{4} \) such that the ball
	\begin{equation}
		\OpenBall_\delta(m) \defeq \set{p \in M \given d(m, p) < \delta}
	\end{equation}
	is contained in \( W_\frac{\varepsilon}{2}(m) \).
	Let \( V_\delta \subseteq U_\frac{\varepsilon}{2}(m,0) \subseteq \NBundle O \) be the inverse image of \( \OpenBall_\delta(m) \) under \( \eta \).
	Hence, by construction, every \( X_p \in V_\delta \) satisfies \( d(m, p) < \frac{\varepsilon}{2} \), \( \rho_p (X_p, 0) < \frac{\varepsilon}{2} \) and \( d(m, \eta(X_p)) < \delta \).

		\begin{figure}
		\centering
		\begin{tikzpicture}
			\draw
					[thick] 
					(-2, 3.6) .. controls (0, 1.5) and (8, -3.5) .. (6, 3.5) 
					node[pos=0.08, circle, fill=black, inner sep=0.2ex](m1){}
					node[pos=0.15, circle, fill=black, inner sep=0.2ex](m){} 
					node[pos=0.22, circle, fill=black, inner sep=0.2ex](m2){} 
					node[pos=0.45, circle, fill=black, inner sep=0.2ex](g1m1){} 
					node[pos=0.65, circle, fill=black, inner sep=0.2ex](g1m){} 
					node[pos=0.85, circle, fill=black, inner sep=0.2ex](g2m){} 
					node[pos=0.95, circle, fill=black, inner sep=0.2ex](g2m2){}; 
			\path (m1) node[above right] {$m_1$};
			\path (m) node[below left] {$m$};
			\path (m2) node[above, xshift=0.4ex, yshift=0.2ex] {$m_2$};
			\path (g1m1) node[below, yshift=-0.4ex] {$g_1\cdot m_1$};
			\path (g1m) node[below] {$g_1\cdot m$};
			\path (g2m) node[below right] {$g_2\cdot m$};
			\path (g2m2) node[below right] {$g_2\cdot m_2$};

			\fill (5, 1) circle [radius=1.5pt] node(intersection){};
			\path (intersection) node[above left] {$x$};
			
			\filldraw[fill=blue1, fill opacity=0.1] (g1m) ellipse[x radius=3.5, y radius=3.5];
			\draw (5.5, -2) node {$W_\varepsilon(g_1\cdot m)$};

			\filldraw[fill=blue2, fill opacity=0.1] (m) circle[radius=1.5];
			\draw (-0.9, 0.5) node {$B_{\frac{\varepsilon}{2}}(m)$};
			
			\draw[dashed, -latex] (g1m1) -- (intersection);
			\draw[dashed, -latex] (g2m2) -- (intersection);
		\end{tikzpicture}
		\caption{Setup of the proof of \cref{prop:slice:sliceTheoremGeneral:diffeoOrbit}.}
		\label{fig:slice:sliceTheoremGeneral:diffeoOrbit}
	\end{figure}
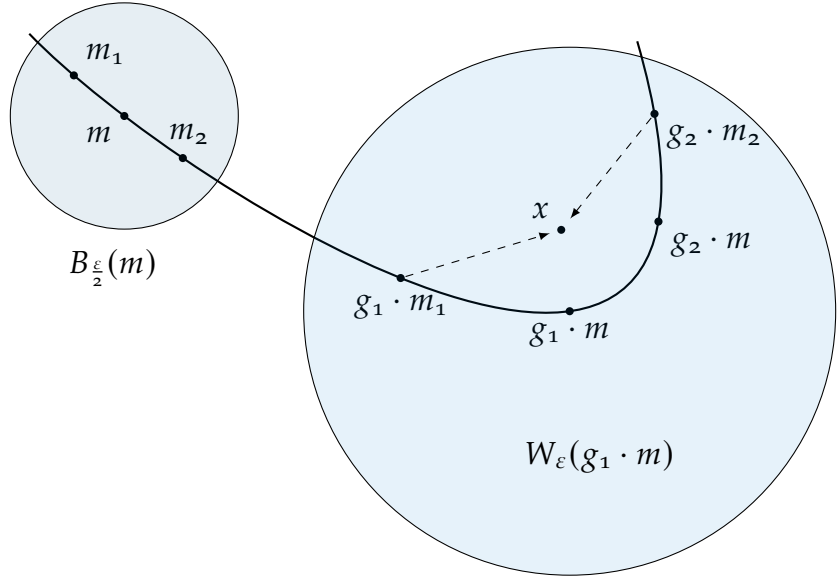
	We claim that \( \eta \) restricts to an injective map on
	\begin{equation}
	 	V \defeq \bigUnion_{g \in G} g \ldot V_\delta \subseteq U \intersect \NBundle O.
	\end{equation}
	For this purpose, let \( g_1, g_2 \in G \) and \( X^1_{m_1}, X^2_{m_2} \in V_\delta \) such that \( \eta(g_1 \ldot X^1_{m_1}) = x = \eta(g_2 \ldot X^2_{m_2}) \).
	This setup is illustrated in \cref{fig:slice:sliceTheoremGeneral:diffeoOrbit}.
	We have
	\begin{equation}\begin{split}
		d(g_2 \cdot m_2, g_1 \cdot m) 
			&\leq d(g_2 \cdot m_2, g_2 \cdot m) + d(g_2 \cdot m, x) + d(x, g_1 \cdot m) 
			\\
			&= d(m_2, m) + d(m, \eta(X^2_{m_2})) + d(\eta(X^1_{m_1}), m) 
			\\
			&< \frac{\varepsilon}{2} + 2 \delta < \varepsilon,
	\end{split}\end{equation}
	because \( \eta \) is \( G \)-equivariant and \( d \) is \( G \)-invariant.
	Moreover, we have 
	\begin{equation}
		\rho_{g_2 \cdot m_2}(g_2 \ldot X^2_{m_2}, 0)
			= \rho_{m_2}(X^2_{m_2}, 0)
			< \frac{\varepsilon}{2}.
	\end{equation}
	Hence, in summary, \( g_2 \ldot X^2_{m_2} \in U_\varepsilon(g_1 \cdot m, 0) \).
	However, by \( G \)-equivariance, \( \eta \) is injective on \( U_\varepsilon(g_1 \cdot m, 0) \) and, thus, \( g_1 \ldot X^1_{m_1} = g_2 \ldot X^2_{m_2} \).
	Thus, \( \eta \) is injective on \( V \), which concludes the proof.
\end{proof}

Finally, we come to the construction of the slice.
\begin{lemma}
	\label{prop:slice:sliceTheoremGeneral:sliceConstruction}
	In the setting of \cref{prop:slice:sliceTheoremGeneral}, there exists \( \varepsilon > 0 \) such that 
	\begin{equation}
		\label{eq:slice:sliceTheoremGeneral:sliceConstruction}
		S \defeq \set{\eta(X_m) \in M \given X_m \in \NBundle_m O, \rho_m (X_m, 0) < \varepsilon}
	\end{equation}
	is a linear slice at \( m \).
\end{lemma}
\begin{proof}
	Let \( \varepsilon > 0 \) be sufficiently small so that the construction in the proof of \cref{prop:slice:sliceTheoremGeneral:diffeoOrbit} yields a \( G \)-invariant open neighborhood \( V \) of the zero section in \( \NBundle O \) with \( V \subseteq U \intersect \NBundle O \) and such that the restriction of \( \eta \) to \( V \) is a \( G \)-equivariant diffeomorphism onto an open neighborhood of \( O \) in \( M \). 
	Let \( S \) be defined as in~\eqref{eq:slice:sliceTheoremGeneral:sliceConstruction}.
	To start with, we note that \( S \) is a submanifold, because \( \NBundle O \) is a subbundle of \( \restr{\TBundle M}{O} \) and the restriction of \( \eta \) to \( U \intersect \NBundle O \) is a local diffeomorphism at \( (m, 0) \).
	Next, we have to verify that \( S \) has all the properties of a linear slice according to \cref{def:slice:slice,def:slice:linearSlice}:
	\begin{enumerate}[label=(SL\arabic*), ref=(SL\arabic*), leftmargin=*]
		\item
			Since \( \eta \) is \( G \)-equivariant and \( \rho \) is \( G \)-invariant, \( S \) is invariant under the action of the stabilizer $G_m$.
		\item
			\( (g \cdot S) \cap S \neq \emptyset \) implies \( g \in G_m \) due to the \( G \)-equivariance of \( \eta \).
		\item
			Let $\chi: G \slash G_m \supseteq W \to G$ be a local section of \( G \to G \slash G_m \) defined on an open neighborhood of the identity coset.
			By properness, the orbit map \( \Upsilon_m \) yields a diffeomorphism between \( G \slash G_m \) and \( G \cdot m \) endowed with the relative topology.
			Since the \( G \)-action is transitive on the orbit, the normal bundle \( \NBundle O \) is trivial.
			Thus, after possibly shrinking \( W \), we may assume that the map
			\begin{equation}
				W \times (U_\varepsilon(m, 0) \intersect \NBundle_m O) \to \NBundle O, \quad (\equivClass{g}, X_m) \mapsto \chi(\equivClass{g}) \ldot X_m
			\end{equation}
			is a diffeomorphism onto an open neighborhood of \( (m, 0) \) in \( \NBundle O \).
			Here, \( U_\varepsilon(m, 0) \) denotes the open \( \varepsilon \)-ball in \( \NBundle O \) centered at \( (m, 0) \) , as in the proof of \cref{prop:slice:sliceTheoremGeneral:diffeoOrbit}.
			By construction, every $s \in S$ is of the form $\eta(X_m)$ with $X_m \in \TBundle_m M$.
			Using the equivariance of \( \eta \), the map $\chi^S: W \times S \to M \) defined by \( \chi^S(\equivClass{g}, s) = \chi(\equivClass{g}) \cdot s$ can be written as
			\begin{equation}
				\chi^S(\equivClass{g}, s)
					= \chi(\equivClass{g}) \cdot \eta(X_m)
					= \eta(\chi(\equivClass{g}) \ldot X_m).
			\end{equation}
			Thus, \( \chi^S \) is a composition of diffeomorphisms and, therefore, it is a diffeomorphism onto an open neighborhood of \( m \) in \( M \).
		\item
			By \cref{prop:properAction:sliceOrbitTypeEqualsSliceStab}, properness of the action implies \( S_{(G_m)} = S_{G_m} \).
			Since \( S \) is \( G_m \)-equivariantly diffeomorphic to an open subset of \( \NBundle_m O \), the set \( S_{G_m} \) is a submanifold of \( S \) if \( (\NBundle_m O)_{G_m} \) is a submanifold of \( \NBundle_m O \).
			Being the fixed point set under the linear \( G_m \)-action, the latter is a closed subspace of \( \NBundle_m O \). 
		\item
			By construction, \( \eta \) is a \( G_m \)-equivariant diffeomorphism from an open neighborhood of \( 0 \) in \( \NBundle_m O \) onto \( S \).
			\qedhere
	\end{enumerate}
\end{proof}

\begin{remark}
	The assumptions of \cref{prop:slice:sliceTheoremGeneral} are automatically satisfied for a proper action of a Banach Lie group \( G \) on Hilbert manifold \( M \) carrying a \( G \)-invariant strong Riemannian metric \( \gamma \).
	Indeed, for every \( m \in M \), the stabilizer \( G_m \) is a compact subgroup of the Banach Lie group \( G \) and hence it is a principal Lie subgroup of \( G \) due to \parencite[Theorem~IV.3.16 and Remark~IV.4.13b]{Neeb2006}.
	The classical inverse function theorem shows that the orbit \( G \cdot m \) is a locally closed submanifold.
	Moreover, since the Riemannian metric \( \gamma \) is \( G \)-invariant, the orthogonal complement of \( \LieA{g} \ldot m \) in \( \TBundle_m M \) yields a realization of the normal bundle of \( G \cdot m \) as a \( G \)-invariant subbundle of \( \TBundle M \).
	Furthermore, the exponential map of \( \gamma \) is a \( G \)-equivariant local addition, which is adapted to \( G \cdot m \) as the inverse function theorem shows.
	Finally, the path length metric associated to \( \gamma \) provides a \( G \)-invariant topological metric \( d \) on \( M \) compatible with the manifold topology.
\end{remark}

\subsection{Slice theorem for Fréchet representations of compact groups}
\label{sec:sliceTheorem:linearActionCompactGroup}

In this section, we will use the general slice \cref{prop:slice:sliceTheoremGeneral} to show that every linear action of a compact group on a Fréchet space admits a slice.
Let \( X \) be a Fréchet space and let \( G \) be a compact Lie group acting linearly on \( X \).
Since \( G \) is compact, the action is proper according to \cref{prop:lieGroupAction:proper:sequenceMetric}.
The other assumptions of \cref{prop:slice:sliceTheoremGeneral} are verified in a sequence of lemmas.
\begin{lemma}
	\label{prop:sliceTheorem:linearActionCompactGroup:principalSubgroup}
	Let \( X \) be a Fréchet space and let \( G \) be a compact Lie group that acts linearly and continuously on \( X \).
	Then, every stabilizer subgroup is a finite-dimensional principal Lie subgroup of \( G \).
\end{lemma}
\begin{proof}
	As a compact Lie group, \( G \) is modeled on a locally compact vector space.
	But every locally compact Hausdorff topological space is finite-dimensional \parencite[Theorem~9.2]{Treves1967}.
	Thus, \( G \) is a finite-dimensional Lie group.
	Since stabilizer subgroups are always closed, they are finite-dimensional principal Lie subgroups according to the ordinary theory for finite-dimensional compact Lie groups.
\end{proof}
More generally, the statement of \cref{prop:sliceTheorem:linearActionCompactGroup:principalSubgroup} is true for every action of a finite-dimensional Lie group on any infinite-dimensional manifold.

The next lemma establishes the existence of invariant metrics on \( X \) and on \( \TBundle X \isomorph X \times X \).
\begin{lemma}
	\label{prop:sliceTheorem:linearActionCompactGroup:invariantMetric}
	Let \( X \) be a Fréchet space and let \( G \) be a compact Lie group that acts linearly and continuously on \( X \).
	Then, there exists a \( G \)-invariant topological metric on \( X \) which induces the same Fréchet topology.
\end{lemma}
\begin{proof}
	For every seminorm \( \normDot \) on \( X \), we define its \( G \)-average by
	\begin{equation}
		\norm{x}^G \defeq \int_G \norm{a \cdot x} \dif a, 
	\end{equation}
	where \( \dif a \) denotes the right-invariant normalized Haar measure on \( G \).
	It is straightforward to see that \( \normDot^G \) is again a seminorm.
	Moreover, \( \normDot^G \) is \( G \)-invariant by construction.
	Indeed, we have
	\begin{equation}
		\norm{g \cdot x}^G 
			= \int_G \norm{(ag) \cdot x} \dif a
			= \int_G \norm{b \cdot x} \dif b
			= \norm{x}^G 
	\end{equation}
	for all \( g \in G \).
	Now, let \( \normDot_k \) be a directed set of seminorms generating the topology of \( X \).
	Since for all \( x \in X \), we have
	\begin{equation}
		\sup_{g \in G} \norm{g \cdot x}_k < \infty
	\end{equation}
	by compactness of \( G \), the uniform boundedness principle (\eg, \parencite[Theorem~10.11]{Simon2017}) implies that for every seminorm \( \normDot_k \) there exists a seminorm \( \normDot_j \) and a constant \( C > 0 \) such that \( \norm{g \cdot x}_k \leq C \norm{x}_j \) holds for all \( g \in G \) and \( x \in X \).
	Hence,
	\begin{equation}
		\norm{x}^G_k = \int_G \norm{a \cdot x}_k \dif a \leq \int_G C \norm{x}_j \dif a = C \norm{x}_j \, .
	\end{equation}
	Similarly, for every seminorm \( \normDot_i \) there exists a seminorm \( \normDot_l \) and a constant \( D > 0 \) such that
	\begin{equation}
		\norm{x}_i = \norm{g^{-1} \cdot (g \cdot x)}_i \leq D \norm{g \cdot x}_l
	\end{equation}
	holds for all \( g \in G \) and \( x \in X \).
	Thus, we get \( \norm{x}_i \leq D \norm{x}^G_l \) and conclude that \( \normDot^G_k \) is a directed set of \( G \)-invariant seminorms equivalent to the original set of seminorms \( \normDot_k \).
	Hence,
	\begin{equation}
		d^G (x, y) \defeq \sum_{k=1}^\infty 2^{-k} \frac{\norm{x-y}^G_k}{1 + \norm{x-y}^G_k}
	\end{equation}
	is a \( G \)-invariant metric compatible with the original topology on \( X \). 
\end{proof}
We exploit the finite-dimensionality of the orbit and use the inverse function theorem of \textcite{Gloeckner2006a} to show that the orbits are submanifolds.
\begin{lemma}
	\label{prop:sliceTheorem:linearActionCompactGroup:orbitSubmanifold}
	Let \( X \) be a Fréchet space and let \( G \) be a compact Lie group that acts linearly and continuously on \( X \).
	Then, for every \( x \in X \), the orbit \( G \cdot x \) is an embedded submanifold of \( X \).
\end{lemma}
\begin{proof}
	Let \( x \in X \). 
	Recall that the orbit map \( \Upsilon_x: G \to X \) descends to an injection \( \check{\Upsilon}_x: G \slash G_x \to X \).
	A straightforward calculation shows that \( \check{\Upsilon}_x \) has an injective differential at the identity coset (see \parencite[Theorem~I]{Gloeckner2015}).
	Since \( G \slash G_x \) is finite-dimensional, \parencite[Theorem~H]{Gloeckner2015} shows that injectivity of the differential is enough to conclude that \( \check{\Upsilon}_x \) is an immersion.
	Moreover, \( G \slash G_x \) is compact and thus \( \check{\Upsilon}_x \) is a topological embedding.
	Consequently, the orbit is an embedded submanifold of \( X \).
\end{proof}

In order to realize the normal bundle of the orbit \( G \cdot x \) as a subbundle of \( \TBundle X \), we need the following preliminary result about invariant complements.
\begin{lemma}
	\label{prop:sliceStratification:invariantComplementCompactGroup}
	Let \( G \) be a compact Lie group which acts linearly and continuously on a Fréchet\footnotetext{The statement holds also in the slightly more general setting when \( X \) is a Mackey complete locally convex vector space.} space \( X \).
	Let \( H \) be a Lie subgroup of \( G \).
	Then, every closed \( H \)-invariant topologically complemented subspace \( E \subseteq X \) admits an \( H \)-invariant complement.
	In particular, for every \( x \in X \), the infinitesimal orbit \( \LieA{g} \ldot x \subseteq X \) has a \( G_x \)-invariant complement.
\end{lemma}
\begin{proof}
	Since \( E \) is complemented, there exists a projection operator onto \( E \), \ie, a continuous left inverse \( \epsilon: X \to E \) of the inclusion \( E \toInject X \).
	Define \( \epsilon_H: X \to E \) as the average 
	\begin{equation}
		\epsilon_H (x) \defeq \int_H h \cdot \epsilon(h^{-1} \cdot x) \, \dif h,
	\end{equation}
	where \( \dif h \) is the normalized Haar measure on \( H \) (note that \( H \) is compact).
	Since \( X \) is complete, the integral exists.
	As \( E \) is \( H \)-invariant, \( \epsilon_H \) takes values in \( E \), indeed.
	Moreover, using continuity of the action, it is easy to see that \( \epsilon_H \) is a continuous linear operator, which is \( H \)-equivariant by construction.
	For all \( e \in E \), we have \( h \cdot \epsilon(h^{-1} \cdot e) = e \) and, thus, \( \epsilon_H \) is an \( H \)-equivariant continuous left inverse of the inclusion \( E \toInject X \). 
	Hence, \( E \) is topologically complemented by the \( H \)-invariant subspace \( \ker \epsilon_H \). 
	The second claim concerning the \( G_x \)-invariant complement of \( \LieA{g} \ldot x \) follows directly from the fact that \( \LieA{g} \ldot x \) is finite-dimensional and that every finite-dimensional subspace is topologically complemented.
\end{proof}
By the above lemma, there exists a \( G_x \)-invariant complement \( Z \) of \( \LieA{g} \ldot x \) in \( X \) for all \( x \in X \).
Setting \( \NBundle_{g \cdot x} (G \cdot x) = g \cdot Z \) yields a representation of the normal bundle of \( G \cdot x \) as a subbundle of \( \TBundle X \isomorph X \times X \).

By \cref{ex:slice:localAddition:linear}, the map
\begin{equation}
	\eta: X \times X \to X, \quad (x, v) \mapsto x + v 
\end{equation}
is a local addition.
The next lemma implies that \( \eta \) is adapted to the submanifold \( G \cdot x \) in the sense of \cref{defn:slice:localAddition:adapted}.
\begin{lemma}
	\label{prop:sliceStratification:localAdditionAdapted}
	Let \( G \) be compact Lie group which acts linearly and continuously on a Fréchet space \( X \).
	For \( x \in X \), let \( Z \) be a \( G_x \)-invariant complement of \( \LieA{g} \ldot x \) in \( X \).
	Then, the map
	\begin{equation}
		\restr{\eta}{\NBundle (G \cdot x)}:  \NBundle (G \cdot x) \to X, \quad (g \cdot x, g \cdot z) \mapsto g \cdot (x + z) 
	\end{equation}
	is a local diffeomorphism at \( (x, 0) \in \NBundle (G \cdot x) \).
\end{lemma}
\begin{proof}
	Choose a \( G_x \)-invariant complement \( \LieA{m} \) of \( \LieA{g}_x \) in \( \LieA{g} \).
	In the sequel, we will identify \( \LieA{g} \ldot x \) with \( \LieA{m} \) and view the latter as a subspace of \( X \) without any further mentioning.
	In particular, \( X \isomorph \LieA{m} \oplus Z \).
	The corresponding projection onto \( \LieA{m} \) will be denoted by \( \pr_{\LieA{m}} \).
	Let \( \kappa: G \slash G_x \supseteq U \to V \subseteq \LieA{m} \) be a chart at the identity coset and let \( \chi: U \to G \) be a local section of \( G \to G \slash G_x \).
	By a slight abuse of notation, we will denote the chart representation \( \chi \circ \kappa^{-1}: \LieA{m} \supseteq V \to G \) by \( \chi \) as well.
	Without loss of generality, we may assume that \( \chi(0) = e \) and that \( \tangent_0 \chi: \LieA{m} \to \LieA{g} \) is the inclusion.
	Clearly, the map
	\begin{equation}
		V \times Z \to \NBundle (G \cdot x), \qquad (\xi, z) \mapsto (\chi(\xi) \cdot x, \chi(\xi) \cdot z)
	\end{equation}
	yields a local trivialization of \( \NBundle (G \cdot x) \) in a neighborhood of \( x \).
	With respect to this local trivialization, the map \( \restr{\eta}{\NBundle (G \cdot x)} \) takes the form
	\begin{equation}
		f: V \times Z \to X, \qquad (\xi, z) \mapsto \chi(\xi) \cdot (x + z)
	\end{equation}
	Since the derivative at \( (0, 0) \) of the map
	\begin{equation}
		V \times Z \to \LieA{m}, \qquad (\xi, z) \mapsto \pr_{\LieA{m}} (\chi(\xi) \cdot x)
	\end{equation}
	 with respect to the first variable is the identity on \( \LieA{m} \) and \( \LieA{m} \) is finite-dimensional, the inverse function theorem of \textcite{Gloeckner2006a} shows that the map
	\begin{equation}
		\psi: V \times Z \to X, \qquad (\xi, z) \mapsto {\pr_{\LieA{m}} \bigl(\chi(\xi) \cdot x\bigl)} + z
	\end{equation}
	is a local diffeomorphism at \( (0, 0) \).
	Furthermore, the map \( \lambda: V \times Z \to V \times Z \) defined by \( \lambda(\xi, z) = \bigl(\xi, z - \pr_Z(\chi(\xi) \cdot x)\bigl) \) is a diffeomorphism.
	Let \( \Psi = \lambda \circ \psi^{-1} \), where restriction to appropriate open neighborhoods is understood.
	Moreover, the map
	\begin{equation}
		\Phi: V \times Z \to V \times Z, \qquad (\xi, z) \mapsto (\xi, \chi(\xi)^{-1} \cdot z)
	\end{equation}
	is a diffeomorphism with inverse given by \( \Phi^{-1}(\xi, z) = (\xi, \chi(\xi) \cdot z) \).
	Now, the calculation
	\begin{equation}\begin{split}
		\Psi \circ f \circ \Phi \bigl(\xi, z\bigr)
			&= \Psi \circ f \bigl(\xi, \chi(\xi)^{-1} \cdot z\bigr)
			\\
			&= \lambda \circ \psi^{-1} \bigl(\chi(\xi) \cdot x + z\bigr)
			\\
			&= \lambda \bigl(\xi, \pr_Z (\chi(\xi) \cdot x) + z \bigr)
			\\
			&= \bigl(\xi, z\bigr).
	\end{split}\end{equation}
	shows that \( f \), and hence also \( \restr{\eta}{\NBundle (G \cdot x)} \), is a local diffeomorphism.
\end{proof}

\Cref{prop:sliceTheorem:linearActionCompactGroup:principalSubgroup,prop:sliceTheorem:linearActionCompactGroup:invariantMetric,prop:sliceTheorem:linearActionCompactGroup:orbitSubmanifold,prop:sliceStratification:invariantComplementCompactGroup,prop:sliceStratification:localAdditionAdapted} verify the assumptions of \cref{prop:slice:sliceTheoremGeneral} for the case under consideration and we thus obtain the following slice theorem.
\begin{thm}[Linear slice theorem]
	\label{prop:slice:sliceTheoremLinearAction}
	Let \( X \) be Fréchet space and let \( G \) be a compact Lie group that acts linearly and continuously on \( X \).
	Then, there exists a linear slice at every point.
\end{thm}

\subsection{Slice theorem for nonlinear tame Fréchet actions}

In the previous section, we have seen that a linear action of a compact Lie group on a Fréchet space always admits a slice.
One of the main ingredients was the finite-dimensionality of the group, which allowed us to use the inverse function theorem of Glöckner.
We now consider the more general situation of nonlinear actions of infinite-dimensional Fréchet Lie groups on Fréchet manifolds in the tame category, that is, we consider tame Fréchet Lie group actions \( (M, G, \Upsilon) \).
In this setting, we are going to verify the assumptions of the general slice \cref{prop:slice:sliceTheoremGeneral} as follows:
\begin{enumerate}[align=left]
	\item
		The condition that the stabilizer \( G_m \) is a principal Lie subgroup remains and needs to be checked in each example separately.
	\item
		To show that the orbits are submanifolds, we use the general inverse function theorem of Nash and Moser formulated in the tame category by \textcite{Hamilton1982}.
		This will reduce assumption (ii) to the requirement that the linearization \( \tangent_e \Upsilon_m: \LieA{g} \to \TBundle_m M \) of the action has a tame (generalized) inverse.
		We emphasize that the existence of an inverse is required only at a single point.
		The family of inverses needed for the Nash--Moser inverse function theorem is obtained by left translation on the group \( G \).
	\item[(iii) and (iv)]
		These assumptions will be verified in the context of infinite-dimen\-sional Riemannian structures on \( M \).
		More precisely, in order for the topology induced by the Riemannian structure on each tangent space to be equivalent to the original Fréchet topology, we invent the notion of a graded Riemannian metric which consists of a family of Riemannian metrics.
		Then, the invariant topological metric on \( M \) is constructed via the path-length distance.
		Moreover, the local addition is obtained via the exponential map of the Riemannian structure.
		However, the existence of the latter is plagued by a number of obstacles, which we are going to discuss below.
\end{enumerate}

For the convenience of the reader, let us recall the main notions of the tame Fréchet category, see \parencite[Part~II]{Hamilton1982} for details.
A Fréchet space \( X \), whose topology is generated by a family \( \normDot_k \) of seminorms, is called graded if \( \normDot_k \leq \normDot_{k+1} \) holds for all \( k \in \N \).
A graded Fréchet space is called \emphDef{tame} if the seminorms satisfy an additional interpolation property, which formalizes the idea that \( X \) admits smoothing operators, see \parencite[Definition~II.1.3.2]{Hamilton1982}.
Let \( X \) and \( Y \) be tame Fréchet spaces. 
A continuous (possibly non-linear) map \( P: X \supseteq U \to Y \) defined on an open subset \( U \subseteq X \) is called \emphDef{tame} if an estimate of the form
\begin{equation}
	\norm{P(x)}_{k} \leq C (1 + \norm{x}_{k+r})
\end{equation}
holds for some \( r \in \N \) in a neighborhood of every point.
Moreover, \( P \) is called \emphDef{tame smooth} if it is smooth and all derivatives \( \dif^{(j)} P: U \times X^j \to Y \) are tame maps.
The notion of a tame manifold is defined in an obvious way.

\subsubsection{Manifold structure of orbits}
We will now discuss under which conditions the orbits are embedded submanifolds.
For this purpose, we we will use the Nash--Moser inverse function theorem in the tame Fréchet category as presented in \parencite{Hamilton1982}.
Recall that, in contrast to its cousin in the Banach setting, this version of the inverse function theorem needs the derivative to be invertible in a whole neighborhood instead of only at a single point.
It is convenient to express this requirement in a differential geometric language using sequences of vector bundles.
Thus, for the moment, consider the abstract setting given by a sequence of vertical vector bundle morphisms
\begin{equationcd}
	E \ar[r, "\phi"] & F \ar[r, "\psi"] & L
\end{equationcd}
over \( M \).
Such a sequence is called \emphDef{exact} at \( F \) if \( \img \phi = \ker \psi \).
If, in addition, the spaces \( \ker \phi_m \), \( \img \phi_m = \ker \psi_m \) and \( \img \psi_m \) are complemented in the fibers \( E_m \), \( F_m \) and \( L_m \), respectively, for all \( m \in M \), then the sequence is called \emphDef{fiber exact}.
Finally, we say that an exact sequence is \emphDef{bundle exact} at \( F \) if \( \ker \phi \), \( \img \phi = \ker \psi \) and \( \img \psi \) are vertical split subbundles\footnote{A subset \( F \subseteq E \) of a vector bundle \( \pi: E \to M \) is called a vertical subbundle if for every point \( f \in F \) there exists a local trivialization \( \tau: \pi^{-1} (U) \to U \times \FibreBundleModel{E} \) of \( E \) at \( \pi(f) \) and a closed subspace \( \FibreBundleModel{F}_U \subseteq \FibreBundleModel{E} \) such that \( \tau( \pi^{-1}(U) \cap F ) = U \times \FibreBundleModel{F}_U \). If all the subspaces \( \FibreBundleModel{F}_U \subseteq \FibreBundleModel{E} \) are complemented, then \( F \) is called a vertical split subbundle.} of \( E \), \( F \) and \( L \), respectively.
An exact sequence \emphDef{splits (tamely)} at \( F \) if there exist (tame) smooth vertical bundle morphisms \( \rho: F \to E \) and \( \lambda: L \to F \) with \( \phi \circ \rho + \lambda \circ \psi = \id_F \).
Finally, a linear tame map \( T: X \to Y \) between tame Fréchet spaces is called \emphDef{tame regular} if the kernel and the image of \( T \) are tamely complemented, that is, 
\begin{equation}
	X = \ker T \oplus \coimg T, \qquad Y = \coker T \oplus \img T,
\end{equation}
and the map \( \pr_{\img T} \circ \restr{T}{\coimg T}: \coimg T \to \img T \) is a tame isomorphism.
Equivalently, there exists a tame generalized inverse of \( T \), that is, a tame linear map \( S: Y \to X \) satisfying \( T \circ S \circ T = T \). 
As a consequence of \parencite[Theorem~II.3.3.3]{Hamilton1982}, elliptic differential operators are tame regular.
We refer the reader to \parencite{DiezThesis} for further details concerning the notion of tame regularity.

Using this terminology, the Nash--Moser theorem yields the following generalization of the well-known fact that smooth maps with injective differentials are locally injections.
\begin{prop}[\textnormal{\cf \parencite[Corollary on p.~542]{AbbatiCirelliEtAl1989}}]
	\label{prop:submanifolds:immersion:byNashMoser}
	Let \( f: M \to N \) be a tame smooth map between tame Fréchet manifolds.
	If the sequence of vector bundles
	\begin{equationcd}[label=eq::submanifolds:immersion:exactSequence]
		0 \to{r} & \TBundle M \to{r}{\tangent f} & f^* \TBundle N
	\end{equationcd}
	over \( M \) is tamely split bundle exact in a neighborhood of \( m \in M \), then \( f \) is a tame immersion at \( m \).
\end{prop}
Using this tool, we can give a general criterion for determining whether the orbits of a Fréchet \( G \)-action are submanifolds.
\begin{prop}
	\label{prop:slice:orbitClosedSubmanifold}
	Let \( (M, G, \Upsilon) \) be a proper tame Fréchet Lie group action and let \( m \in M \).
	Assume that the stabilizer \( G_m \) is a principal tame Fréchet Lie subgroup of \( G \).
	If the map \( \tangent_e \Upsilon_m: \LieA{g} \to \TBundle_m M \) is tame regular, then the orbit \( G \cdot m \) is a closed split submanifold of \( M \).
\end{prop}
\begin{proof}
	Recall that, for every \( m \in M \), the action \( \Upsilon_m: G \to M \) induces an injective map \( \check{\Upsilon}_m: G \slash G_m \to M\).
	Since G acts properly, \cref{prop:lieGroupProperAction:properties} implies that the orbit $G \cdot m$ is closed in $M$ and that \( \check{\Upsilon}_m \) is a homeomorphism onto \( G \cdot m \). 
	Due to \cref{prop:submanifolds:immersion:byNashMoser}, the orbit \( G \cdot m \) is an embedded submanifold of \( M \) if the sequence
	\begin{equationcd}[label=eq:slice:orbitClosedSubmanifold:tangentSequence]
		0 \to{r}
			& \TBundle (G \slash G_m) \to{r}{\tangent \check{\Upsilon}_m}
			& \check{\Upsilon}_m^* \TBundle M
	\end{equationcd}
	over \( G \slash G_m \) is tamely split bundle exact in a neighborhood of the identity coset.
	Let us verify the required properties step by step.
	\begin{itemize}
		\item Tame smoothness:
			As $G_m$ is a principal tame Lie subgroup, the quotient \( G \slash G_m \) is a tame Fréchet manifold such that the canonical projection \( \pi_{G_m}: G \to G \slash G_m \) admits tame smooth local sections.
			The map $\check\Upsilon_m$ is tame smooth, because it locally factors into the tame smooth map $\Upsilon_m$ and a tame smooth section of $\pi_{G_m}$.

		\item Injectivity of $\tangent \check{\Upsilon}_m$:
			Due to the factorization \( \Upsilon_m = \check\Upsilon_m \circ \pi_{G_m} \), the differential \( \tangent_{\equivClass{g}} \check{\Upsilon}_m \) at \( \equivClass{g} \in G \slash G_m \) is injective if $\ker \tangent_g \Upsilon_m = \ker \tangent_g \pi_{G_m}$ holds.
			By equivariance of $\Upsilon_m$ and $\pi_{G_m}$, it suffices to consider the case $g = e$.
			Since $G_m$ is a principal Lie subgroup, there exists a smooth local section $\sigma: V \to G$ of \( \pi_{G_m} \) defined on an open neighborhood $V$ of the identity coset such that the map \( \mu: V \times G_m \to G \) defined by \( \mu(X, h) = \sigma(X)h \) is a diffeomorphism onto an open neighborhood of the identity in $G$.
			Thus, every smooth curve $\gamma$ in \( G \) with $\gamma(0) = e$ induces smooth curves $\gamma_V$ and $\gamma_{G_m}$ in $V$ and $G_m$, respectively, for sufficiently small $t$.
			That is, $\gamma(t) = \sigma(\gamma_V(t)) \gamma_{G_m}(t)$.
			Thus, if \( \gamma \) is a curve representing \( \xi \in \TBundle_e G \), we obtain 
			\begin{equation}\begin{split}
				\tangent_e \Upsilon_m(\xi)
					&= \difFracAt{}{t}{t=0} (\Upsilon_m \circ \gamma)(t)
					= \difFracAt{}{t}{t=0} \bigl(\sigma(\gamma_V(t)) \gamma_{G_m}(t)\bigr) \cdot m
					\\
					&= \difFracAt{}{t}{t=0} \sigma(\gamma_V(t)) \cdot m
					= \tangent_e \sigma (\dot{\gamma}_V) \ldot m \,.
			\end{split}\end{equation}
			By definition, $\sigma(\gamma_V(t))$ is never an element of \( G_m \) (except when it is the identity element) and hence \( \tangent_e \Upsilon_m (\xi) \) vanishes only if $\dot{\gamma}_V = 0$.
			On the other hand, $\pi_{G_m}$ corresponds in the product coordinates given by \( \mu \) to the projection on the $V$-component and thus $\tangent_e \pi_{G_m} (\xi) = \dot\gamma_V$.
			Consequently, \( \ker \tangent_g \Upsilon_m = \ker \tangent_g \pi_{G_m} \) and \( \tangent_{\equivClass{g}} \check{\Upsilon}_m \) is injective for all \( \equivClass{g} \in G \slash G_m \).
		\item Bundle structure of $\img (\tangent \check{\Upsilon}_m)$:
			By equivariance, it suffices to construct the subbundle chart over some neighborhood of the identity coset.
			Exploiting the local triviality of \( \pi_{G_m} \), we can again work with \( \Upsilon_m \) instead of \( \check{\Upsilon}_m \).
			That is, it is enough to show that the bundle
			\begin{equation}
				\img (\tangent \Upsilon_m) = \bigDisjUnion_{g \in G} \img (\tangent_g \Upsilon_m)
			\end{equation}
			is a subbundle of \( \Upsilon^*_m \TBundle M \) in a neighborhood of the identity.
			Let \( U \subseteq G \) be an open neighborhood of the identity such that
			\begin{equation}
				\label{eq:slice:orbitClosedSubmanifold:trivalization}
				\tau: U \times \TBundle_m M \to \restr{(\Upsilon_m^* \TBundle M)}{U},
				\qquad
				(g, X) \mapsto (g, g \ldot X)
			\end{equation}
			is a local trivialization of \( \Upsilon_m^* \TBundle M \).
			Due to the equivariance relation $g \cdot \Upsilon_m(a) = \Upsilon_m (ga)$, tangent vectors of the form $X = \tangent_g \Upsilon_m (g \ldot \xi)$ with $\xi \in \LieA{g}$ are mapped to $(g, \tangent_e \Upsilon_m (\xi))$ under $\tau^{-1}$.
			Therefore, $\tau^{-1}(\img (\tangent \Upsilon_m)) = U \times \img (\tangent_e \Upsilon_m)$.
			Since \( \tangent_e \Upsilon_m: \LieA{g} \to \TBundle_m M \) is tame regular, its image is tamely complemented.
			Hence, \( \img (\tangent \Upsilon_m) \) is a split subbundle.
		\item Tame left inverse:
			Since  \( \tangent_e \Upsilon_m: \LieA{g} \to \TBundle_m M \) is tame regular, there exists tame decompositions
			\begin{equation}
				\LieA{g} = \LieA{m} \oplus \LieA{g}_m, \qquad \TBundle_m M = \LieA{g} \ldot m \oplus Y,
			\end{equation}
			where \( \LieA{m} \) and \( Y \) are closed subspaces of \( \LieA{g} \) and \( \TBundle_m M \), respectively.
			By averaging over \( G_m \), we may assume that \( \LieA{m} \) is \( \AdAction_{G_m} \)-invariant, see \cref{prop:sliceTheorem:linearActionCompactGroup:invariantMetric}.
			Moreover, tame regularity implies that \( \tangent_e \Upsilon_m \) induces a tame isomorphism from \( \LieA{m} \) onto \( \LieA{g} \ldot m \).
			In particular, there exists a tame left inverse of \( \restr{(\tangent_e \Upsilon_m)}{\LieA{m}} \).

			Now, consider the vector bundle map \( \tangent \check{\Upsilon}_m: \TBundle (G \slash G_m) \to \check{\Upsilon}_m^* \TBundle M \) over \( G \slash G_m \).
			Since \( \pi_{G_m}: G \to G \slash G_m \) is a locally trivial tame principal \( G_m \)-bundle, there exists a tame smooth local section \( \chi: G \slash G_m \supseteq V \to G \) defined on a neighborhood \( V \) of the identity coset.
			The section \( \chi \) induces a local trivialization of \( \TBundle (G \slash G_m) \) by the map
			\begin{equation}
				V \times \LieA{m} \to \TBundle (G \slash G_m), \qquad \bigl(\equivClass{g}, \xi\bigr) \mapsto \tangent_{\chi(\equivClass{g})} \pi_{G_m} \bigl( \chi(\equivClass{g}) \ldot \xi \bigr).
			\end{equation}
			Similarly, as mentioned above, we obtain a local trivialization of \( \check{\Upsilon}^*_m \TBundle M \) by
			\begin{equation}
				\check{\tau}: V \times \TBundle_m M \to \restr{(\check{\Upsilon}_m^* \TBundle M)}{V},
				\qquad
				\bigl(\equivClass{g}, X\bigr) \mapsto \bigl(\equivClass{g}, \chi(\equivClass{g}) \ldot X\bigr).
			\end{equation}
			The calculation
			\begin{equation}\begin{split}
				\tangent_{\equivClass{g}} \check{\Upsilon}_m \circ \tangent_{\chi(\equivClass{g})} \pi_{G_m} \bigl( \chi(\equivClass{g}) \ldot \xi \bigr)
					&= \tangent_{\chi(\equivClass{g})} \Upsilon_m \bigl( \chi(\equivClass{g}) \ldot \xi \bigr)
					\\
					&= \chi(\equivClass{g}) \ldot \tangent_e \Upsilon_m (\xi)
					\\
					&= \chi(\equivClass{g}) \ldot (\xi \ldot m)
			\end{split}\end{equation}
			shows that, with respect to these trivializations, the bundle map \( \tangent \check{\Upsilon}_m \) is given by the map
			\begin{equation}
				V \times \LieA{m} \to V \times \TBundle_m M, \qquad \bigl(\equivClass{g}, \xi\bigr) \mapsto \bigl(\equivClass{g}, \xi \ldot m\bigr).
			\end{equation}
			Hence, the tame left inverse of \( \restr{(\tangent_e \Upsilon_m)}{\LieA{m}} \) induces a tame smooth family of left inverses of \( \tangent \check{\Upsilon}_m \) in the neighborhood \( V \) of the identity coset.
			\qedhere
	\end{itemize}
\end{proof}
We emphasize that we only require a (generalized) inverse of the map \( \tangent_e \Upsilon_m \) and not a family of inverses although the Nash--Moser inverse function theorem is employed.
As we have seen in the proof, the family of inverses is obtained by left translation on the group \( G \).

Under the same assumptions as in the previous proposition, let \( N_m \) be a tame complement of \( \img (\tangent_e \Upsilon_m) \) in \( \TBundle_m M \).
Then, the normal bundle
\begin{equation}
	\NBundle O \defeq \bigDisjUnion_{g \in G} g \ldot N_m
\end{equation}
of the orbit $O = G\cdot m$ is a tame smooth split subbundle of \( \restr{\TBundle M}{O} \).
Indeed, the image of \( \NBundle O \) under the local trivialization \( \tau \) defined in~\eqref{eq:slice:orbitClosedSubmanifold:trivalization} is \( U \times N_m \) and thus \( \NBundle O \) is a tame smooth subbundle of \( \restr{\TBundle M}{O} \) complementary to \( \TBundle O \).

\begin{remark}
	\Textcite{JotzNeeb2008} have shown that for proper actions of Banach Lie groups on Banach manifolds the infinitesimal orbit \( \LieA{g} \ldot m \) is automatically closed in \( \TBundle_m M \).
	We do not know whether this result generalizes to proper (tame) Fréchet Lie group actions.
\end{remark}

\subsubsection{Riemannian geometry} \label{cha::lcm:gradedRiemannianGeometry}
In this subsection we consider \( G \)-manifolds endowed with Riemannian structures.
As outlined above, the Riemannian structure will be used to construct the topological metric on \( M \) and the equivariant local addition.
Some of the following results can be found in \parencite{Subramaniam1984}, but are included with proof for the convenience of the reader.

In the Banach space setting, the theory of Riemannian geometry splits into two branches depending on whether the map \( \TBundle M \to \CotBundle M \) induced by the metric is an isomorphism or merely an injection.
The former are called strong metrics and the latter are referred to as weak metrics. 
Strong Riemannian metrics induce a topology on the tangent spaces equivalent to the one induced from the manifold topology, while for weak Riemannian metrics the induced topology is coarser. 
The well-consolidated building of finite-dimensional Riemannian geometry generalizes almost without changes to strong metrics, but it is confronted with serious problems if weak metrics are considered.
One of the main complications for weak metrics is that the Koszul formula no longer establishes the existence of the Levi--Civita connection. 

In the case of a Fréchet manifold, a Riemannian metric cannot be strong, because the dual of a Fréchet space is never a Fréchet space (except when it is a Banach space).
Our main idea is to circumvent some of these problems by using not only one metric but a whole collection of them which are strong in the sense that together they induce an equivalent topology on the tangent space.
This approach is inspired by the fact that the topology of a Fréchet space is generated by a family of seminorms.

\begin{defn}
	Let \( M \) be a Fréchet manifold.
	A \emphDef{graded Riemannian metric} is a family of smooth functions \( g^k: \TBundle M \times_M \TBundle M \to \R \) indexed by \( k \in \N_0 \) with the following properties: 
	\begin{thmenumerate}
		\item 
			For all \( m \in M \), the functions \( g^k_m: \TBundle_m M \times \TBundle_m M \to \R \) are semi-inner products, that is, they are symmetric, bilinear and positive semi-definite.  
		\item
			The induced seminorms\footnote{In the following, no powers of seminorms are required and thus the notation $\normDot^k$ should not lead to confusions.} $\normDot^k_m = \sqrt{g^k_m(\cdot, \cdot)}$ constitute a directed fundamental system\footnote{A set \( \normDot^k \) of continuous seminorms on a locally convex space \( X \) is called a fundamental system if for every continuous seminorm \( \absDot \) on \(  X \) there exists \( k \) and \( C > 0 \) such that \( \absDot \leq C \normDot^k \). A fundamental system is directed if \( \normDot^k \leq \normDot^{k+1} \) for all \( k \).} generating a topology on \( \TBundle_m M \) equivalent to the original one for all \( m \in M \).
			If \( M \) is a tame Fréchet manifold, we require \( \normDot^k \) to be tamely equivalent to the original fundamental system generating the topology on \( \TBundle_m M \).
			\qedhere
	\end{thmenumerate}	 
\end{defn}
The definition requires only a weak cohesion of the collection $g^k$ at two nearby points.
Indeed, consider the local representation \( g^k: U \times X \times X \to \R \) of \( g^k \) in some chart at \( m \) with \( U \subseteq X \) being an open neighborhood of \( 0 \) in the local model space \( X \) of \( M \).
For all \( x \in U \), the family of semi-inner products \( g^k_x \) induces a topology on \( X \) equivalent to the one induced by \( g^k_0 \).
Hence, for every \( k \), there exists \( n \in \N \) and \( C > 0 \) such that \( \normDot^k_x \leq C \normDot^n_0 \) (and vice-versa with  the roles of \( x \) and \( 0 \) swapped).
Note that the index \( n \) might not only depend on \( k \) but also on the point \( x \).
For our purposes, a more uniform control is advantageous.
The following definition is inspired by the notion of tame equivalence of sets of seminorms, \cf \parencite[Definition~II.1.1.3]{Hamilton1982}.
\begin{defn}
	\label{def:riemannianGeometry:locallyEquivalentMetric}
	Let \( M \) be a Fréchet manifold.
	A graded Riemannian metric \( g^k \) on \( M \) is called \emphDef{locally equivalent} of degree \( r \) if for every $m \in M$ and $K>1$ there exists a chart $\kappa: M \supseteq U \to X$ at $m$ such that
	\begin{equation}
		\label{eq:riemannianGeometry:locallyEquivalentMetric}
  		\norm{v}^k_x \leq K \norm{v}^{k+r}_0
  		\quad \text{and} \quad
  		\norm{v}^k_0 \leq K \norm{v}^{k+r}_x
	\end{equation}
	holds for all \( k \) and $x \in U$, \( v \in X \).
\end{defn}
\begin{example}
	Let \( M \) be a compact finite-dimensional manifold.
	Consider the set \( \MetricSpace \) of Riemannian metrics on \( M \), which is an open subset of the Fréchet space \( \SymTensorFieldSpace^2 \) of symmetric \( 2 \)-tensors.
	Following \textcite{Ebin1970}, endow \( \TBundle \MetricSpace \isomorph \MetricSpace \times \SymTensorFieldSpace^2 \) with the following family of semi-inner products:
	\begin{equation}
		\dualPairR{h_1}{h_2}^k_g = \sum_{j=0}^k \,\, \int_M g(\nabla^j_g h_1, \nabla^j_g h_2) \vol_g, 	
	\end{equation}
	where \( \nabla^j_g \) denotes the iterated covariant derivative induced by the metric \( g \).
	In \parencite[p.~21]{Ebin1970}, Ebin has shown that \( \dualPairRDot^k_g \) is continuous in \( g \) and that for all \( g \in \MetricSpace \) there exists \( C_g > 0 \) such that
	\begin{equation}
		\norm{h}_{\HFunctionSpace^k} \leq C_g \norm{h}^k_g,
	\end{equation}
	where \( \normDot_{\HFunctionSpace^k} \) are the usual Sobolev norms on \( \SymTensorFieldSpace^2 \).
	Moreover, the constant \( C_g \) is given in terms of the Christoffel symbols of \( g \) and so it depends continuously on \( g \).
	Thus, \( \dualPairRDot^k \) is a locally equivalent graded Riemannian metric on \( \MetricSpace \).
\end{example}

Analogously to the finite-dimensional case, we define the \emphDef{length of a \( \FunctionSpace^1 \)-curve} $\gamma: \R \supseteq [a, b] \to M$ with respect to the $k$-th component as
\begin{equation}
	l^k (\gamma) \defeq \int_a^b \norm{\dot \gamma(t)}^k_{\gamma(t)} \dif t.
\end{equation}
Since \( \normDot^k \) is a directed system, we have \( l^k \leq l^{k+1} \) for all \( k \).
If two points $p$ and $q$ lie in the same connected component of $M$, then their (geodesic) distance $d^k(p,q)$ is defined as the infimum of $l^k$ over all piecewise \( \FunctionSpace^1 \)-curves connecting $p$ and $q$.
Finally, we define the length metric as
\begin{equation}
	\label{eq:riemannian:lengthMetric}
	d(p,q) \defeq \sum_{k=1}^\infty 2^{-k} \frac{d^k(p,q)}{1 + d^k(p,q)}.
\end{equation}
For a finite-dimensional Riemannian manifold the geodesic distance is always positive, because the exponential map is locally invertible.
However, this is no longer true in infinite dimensions and it might happen that the geodesic distance function \( d^k \) vanishes for some \( k \).
For example, \textcite[Section~5]{MichorMumford2005} show that the \( \LTwoFunctionSpace \)-geodesic distance on the identity component of the diffeomorphism group vanishes, that is, every two diffeomorphisms isotropic to the identity can be connected by a curve which is arbitrarily short with respect to the \( \LTwoFunctionSpace \)-metric.
However, the next proposition shows that, for a locally equivalent graded Riemannian metric, \( d^k \) cannot vanish for \emph{all} \( k \).
\begin{prop}[{\cf \parencite[p. 52]{Subramaniam1984}}]
	\label{prop::riemannianGeometry:locEquivalentMetricInducesLengthMetric}
	Let $M$ be a Fréchet manifold and let $g^k$ be a graded Riemannian metric, which is locally equivalent with degree \( r \).
	Then, the function $d$ given by~\eqref{eq:riemannian:lengthMetric} is a metric on each connected component of $M$ compatible with the original manifold topology.
\end{prop}
\begin{proof}
	All properties of a metric are evident, except positivity.
	Fix a point $p \in M$ and let \( U \) be an open neighborhood of \( p \) in \( M \).
	Let \( q \in M \) and let \( \gamma \) be an arbitrary \( \FunctionSpace^1 \)-curve connecting \( p \) to \( q \).
	We have to show that \( l^k(\gamma) \) is bounded away from \( 0 \) by a constant that does not depend on \( \gamma \).
	It suffices to consider the case where the curve $\gamma$ completely lies in $U$.
	To see this, suppose $\img(\gamma) \nsubseteq U$.
	By continuity of $\gamma$, there exists $c \in [a,b]$ such that for all parameters \( t \leq c \) the curve takes values in $U$ and $\gamma(c) \neq \gamma(a)$. But $l^k(\gamma) \geq l^k(\gamma_{\restriction_{[a,c]}})$ implies that it is enough to show positivity of \( l^k \) for curves lying completely in \( U \).

	Thus, assume now that the image of \( \gamma \) is contained in \( U \).
	Using a chart on \( M \), we identify \( U \) with an open subset \( V \) of \( 0 \) in the modeling Fréchet space \( X \).
	Thus, under the usual identifications, $\gamma$ is a curve $\gamma: [a,b] \to V \subseteq X$ from \( 0 \) to \( q \in V \) and the metric is a function $g^k: V \times X \times X \to \R$.
	Due to the local equivalence property~\eqref{eq:riemannianGeometry:locallyEquivalentMetric}, we can assume that, for some \( K > 1 \),
	\begin{equation}
  		\normDot^k_x \leq K \normDot^{k+r}_0
  		\quad \text{and} \quad
  		\normDot^k_0 \leq K \normDot^{k+r}_x
	\end{equation}
	hold for all $x \in U$.
	Using the mean value theorem and the fundamental theorem of calculus \parencite[Theorems~2.1.1 and~2.2.2]{Hamilton1982}, the length of $\gamma$ can be estimated from below by
	\begin{equation}\label{eq:riemannianGeometry:lengthMetricEstimateLengthOfCurve} \begin{split} 
		l^k(\gamma) 
			&= \int_a^b \norm{\dot \gamma(t)}^k_{\gamma(t)} \dif t
			\geq \frac{1}{K} \int_a^b \norm{\dot \gamma(t)}^{k-r}_0 \dif t
			\\
			&\geq \frac{1}{K} \norm*{\int_a^b \dot \gamma(t) \dif t}^{k-r}_0
			= \frac{1}{K} \norm{\gamma(b) - \gamma(a)}^{k-r}_0
			= \frac{1}{K} \norm{q}^{k-r}_0,
	\end{split}\end{equation}
	for all \( k \geq r \).
	Since fundamental systems of seminorms separate points, there exists $n$ such that $\norm{q}^n_0$ is non-zero. 
	Thus, $l^k$ with \( k = n + r \) is bounded away from \( 0 \) by a positive constant not depending on \( \gamma \). 
	Therefore, $d(p,q) > 0$ whenever $p\neq q$ and $d$ is a metric.

	Finally, it remains to show that the topology of $d$ coincides with the manifold topology.
	Since charts on \( M \) are homeomorphisms, the argument at the beginning of the proof shows that it is enough to restrict attention to the linear setting.
	Thus, let \( V \subseteq X \) be an open subset of the modeling Fréchet space \( X \).
	Let \( (q_i) \) be a sequence in \( V \).
	As \( \normDot^k_0 \) is a compatible system of seminorms, it suffices to show that \( \norm{q_i}^k_0 \to 0 \) for all \( k \) is equivalent to \( d^n(0, q_i) \to 0 \) for all \( n \).
	Since~\eqref{eq:riemannianGeometry:lengthMetricEstimateLengthOfCurve} implies \( \norm{q_i}^k_0 \leq K d^{k+r}(0, q_i) \) one direction is clear.
	For the converse direction, suppose \( \norm{q_i}^k \to 0 \) for all \( k \).
	The path-length can be estimated by considering the linear curve $\sigma_i(t) = t q_i$ with \( t \in [0,1] \).
	Then, using the local equivalence of the seminorms, we get
		\begin{equation}
		 	d^k(0, q_i)
		 		\leq l^k(\sigma_i)
		 		= \int_0^1 \norm{\dot \sigma_i(t)}^k_{\sigma_i(t)} \dif t
		 		= \int_0^1 \norm{q_i}^k_{\sigma_i(t)}
		 		\dif t \leq K \norm{q_i}^{k+r}_0
		\end{equation}
	and the claim follows immediately.
\end{proof}

Using similar arguments, one can prove metrizability of a wide class of Lie groups.
\begin{prop}[\textnormal{\cf \parencite[p. 53]{Subramaniam1984}}]
	\label{prop:lcLieGroup:metrisableZeroTame}
	Every tame Fréchet Lie group \( G \) has a compatible left (or right) invariant metric.
\end{prop}
\begin{proof}
	Let \( \normDot^k \) be a directed fundamental system of seminorms on the Lie algebra \( \LieA{g} \) of \( G \) generating the Fréchet topology.
	Using the left trivialization \( \TBundle G \isomorph G \times \LieA{g} \), we can extend \( \normDot^k \) to every fiber of \( \TBundle G \) by setting
	\begin{equation}
		\norm{g \ldot \xi}^k_g \defeq \norm{\xi}^k \qquad g \in G, \xi \in \LieA{g}.
	\end{equation}
	Since left translation is tame, the resulting family of seminorms is locally equivalent in the sense that an estimate of the form~\eqref{eq:riemannianGeometry:locallyEquivalentMetric} holds. 
	Now the claim follows from the observation that the proof of \cref{prop::riemannianGeometry:locEquivalentMetricInducesLengthMetric} does not rely on the inner product structure and, instead, only the induced locally equivalent seminorms are required.
\end{proof}

In the light of these results, it is natural to define the exponential map via its distance-minimizing property.
\begin{defn}
	\label{defn:riemannianGeometry:exponentialMap}
	Let $g^k$ be a graded Riemannian metric on the manifold \( M \).
	For \( r \in \N \), an \emphDef{$r$-exponential map} for $g^k$ is a smooth map $\exp: \TBundle M \supseteq U \to M$ defined on an open convex subset $U$ of the zero section in $\TBundle M$ such that the following properties hold:
	\begin{thmenumerate}
		\item 
	 		\( \exp(m, 0) = m \) for all \( m \in M \).
	 	\item 
	 		\( \exp_m \equiv \restr{\exp}{U \intersect \TBundle_m M} \) is a local diffeomorphism for all \( m \in M \).
	 	\item 
	 		For every \( X \in U \intersect \TBundle_m M \), the associated curve $\lambda_X: [0,1] \ni t \mapsto \exp(t X)$ fulfills $\difFracAt{}{t}{t=0} \lambda_X(t) = X$ and is the $r$-shortest path between its endpoints, that is
			\begin{equation}
				l^r(\lambda_X) = d^r(m, \exp(X)). \qedhere
			\end{equation}
	\end{thmenumerate}
\end{defn}
For our purposes, the most important consequence is that an exponential map yields a local addition.
\begin{prop}
	\label{prop:riemannian:expGivesLocalAddition}
	Let $g^k$ be a graded Riemannian metric on the manifold \( M \).
	Every $r$-exponential map \( \exp: \TBundle M \supseteq U \to M \) for a graded Riemannian metric $g^k$ is a local addition on \( M \) (after possibly shrinking \( U \)).
\end{prop}
\begin{proof}
	Let \( m \in M \).
	Since \( \exp_m: U \intersect \TBundle_m M \to M \) is a local diffeomorphism at \( 0 \), there exists an convex open neighborhood \( U_m \) of \( 0 \) in \( \TBundle_m M \) and an open neighborhood \( V_m \) of \( m \) in \( M \) such that the restriction of \( \exp_m \) to \( U_m \) is a diffeomorphism onto \( V_m \).
	By shrinking \( U \), we may assume \( U = \bigUnion_{m \in M} U_m \) (which still is an open neighborhood of the zero section in \( \TBundle M \)).
	Note that
	\begin{equation}
		V = \set{(m, p) \in M \times M \given p \in V_m}
	\end{equation}
	is an open neighborhood of the diagonal in \( M \times M \).
	Now, it is straightforward to see that the map
	\begin{equation}
		\pr_M \times \exp: \TBundle M \supseteq U \to V, \qquad (m, X) \mapsto (m, \exp_m (X))
	\end{equation}
	is a diffeomorphism with smooth inverse given by \( (m, p) \mapsto \bigl(m, \exp_m^{-1}(p)\bigr) \).
	Moreover, we clearly have \( \exp(m, 0) = m \) for all \( m \in M \).
	In summary, \( \exp: \TBundle M \supseteq U \to M \) is a local addition on \( M \).
\end{proof}

\begin{remark}
	The proof of \cref{prop:riemannian:expGivesLocalAddition} shows that the distance-minimizing property of the Riemannian exponential map is not required in order to get a local addition.
	Thus, more generally, every spray on \( M \) for which the integral curves exist and are unique yields a local addition if the exponential map defined by the spray is a local diffeomorphism (which is automatic in the Banach setting, see \parencite[Theorem~IV.4.1]{Lang1999}).
\end{remark}

Given a graded Riemannian metric, there could exist none or many exponential maps conforming to the above definition.

In finite dimensions, the exponential map is constructed as follows:
\begin{enumerate}
	\item
		Calculate the Levi--Civita connection using the Koszul formula.
	\item
		Establish existence and uniqueness of solutions for the geodesic equation (and show that solutions depend smoothly on the initial data).
	\item
		Show that the exponential map is a local diffeomorphism.
	\item
		Verify that geodesics are distance-minimizing.
\end{enumerate}
All of the above steps are confronted with problems in infinite dimensions.
First, the Koszul formula ensures only uniqueness and no longer the existence of the Levi--Civita connection, because in general the metric is not strongly non-degenerate.
Moreover, existence and uniqueness of solutions of the initial value problem for ordinary differential equations in Fréchet spaces is no longer guaranteed (see, \eg, \parencite[Appendix~A.4]{Michor2006} for counterexamples).
Hence, it may happen that there exists no solution to the geodesics equation and hence also no Riemannian exponential map.
Even if \( \exp \) exists it does not need to be a local diffeomorphism.
This was shown, for example, by \textcite{ConstantinKolev2002} for the exponential map of the \( \LTwoFunctionSpace \)-metric on the group of orientation-preserving diffeomorphisms of the circle.
Finally, the work of \textcite[Section~5]{MichorMumford2005} shows that geodesics are not necessarily locally distance-minimizing paths.

Nonetheless, these problems often only occur for metrics of Sobolev type with low regularity.
For example, for \( k \geq 1 \), the \( H^k \)-metric on the group of orientation preserving diffeomorphisms of the circle has an exponential map in the sense of \cref{defn:riemannianGeometry:exponentialMap}, see \parencite{ConstantinKolev2003} for details.
We refer the reader to \parencite{BauerBruverisEtAl2013} for a recent review concerning well-posedness of the geodesic equation and positivity of the geodesic distance in various examples of geometric interest.

\begin{remark}
	It is an open question how exponential maps corresponding to different values of $r$ are related to each other.
	Examples suggest that the existence of an $r$-exponential map is an indicator for the existence of all \( r' \)-exponential maps for all \( r' > r \).
	This phenomenon is closely related to the \enquote{regularity of geodesics} \parencite{EbinMarsden1970}, that is, geodesics relative to one Riemannian metric $g^r$ transform under suitable extra conditions into geodesics with respect to another $g^{r'}$.
\end{remark}

\subsubsection{Construction of the slice}
With this preparation at hand, we can now use \cref{prop:slice:sliceTheoremGeneral} to construct the slice.
\begin{thm}[Fréchet slice theorem]
	\label{prop::liegroup:sliceTheorem}
	Let \( (M, G, \Upsilon) \) be tame Fréchet Lie group action.
	The $G$-action admits a linear slice at the point \( m \in M \) if the following conditions are fulfilled:
	\begin{enumerate}
		\item
			The \( G \)-action is proper.
		\item
			The stabilizer $G_{m}$ is a principal tame Fréchet Lie subgroup of $G$.
		\item
			The linearized action \( \tangent_e \Upsilon_m: \LieA{g} \to \TBundle_m M \) is a tame regular map.
		\item
			$M$ carries a $G$-invariant, locally equivalent, graded Riemannian metric $g^k$ such that the $r$-exponential map \( \exp \) exists for some $r$.
			Furthermore, assume that the restriction of \( \exp \) to the normal bundle \( \NBundle O \) of the orbit \( O = G \cdot m \) is an equivariant local diffeomorphism at every point of the zero section.
			\qedhere
	\end{enumerate}
\end{thm}
We emphasize that only a (generalized) inverse of the map \( \tangent_e \Upsilon_m \) is required in \cref{prop::liegroup:sliceTheorem}, although the Nash--Moser inverse function theorem needs a tame family of inverses.
\begin{proof}
	By \cref{prop:slice:orbitClosedSubmanifold}, the orbit $O = G \cdot m$ is a split embedded submanifold of \( M \) and the normal bundle \( \NBundle O \) is a smooth subbundle of \( \restr{\TBundle M}{O} \).
	According to \cref{prop:riemannian:expGivesLocalAddition}, the exponential map is a local addition.
	The local addition is adapted to the orbit, because the restriction of \( \exp \) to the normal bundle \( \NBundle O \) is an equivariant local diffeomorphism at every point of the zero section by assumption.
	Furthermore, the graded metric $g^k$ induces pointwise seminorms which combine to a $G$-invariant, compatible metric $\rho_m$ on $\TBundle_m M$.
	\Cref{prop::riemannianGeometry:locEquivalentMetricInducesLengthMetric} yields a length-metric $d$ on $M$ which is compatible with the manifold topology.
	The metric $d$ inherits the $G$-invariance of $g^k$ by inspection of the proof of \cref{prop::riemannianGeometry:locEquivalentMetricInducesLengthMetric}.
	Thus, all the assumptions of \cref{prop:slice:sliceTheoremGeneral} are satisfied and so a slice at \( m \) exists.
\end{proof}

\subsection{Slice theorem for product actions}
\label{sec:sliceTheorem:productActions}

Let \( M \) and \( N \) be \( G \)-manifolds.
Consider the diagonal \( G \)-action
\begin{equation}
	g \cdot (m, n) = (g \cdot m, g \cdot n)
\end{equation}
on \( M \times N \).
We will now show how to construct a slice for this action.
\begin{prop}
	\label{prop:slice:sliceForProduct}
	Let \( G \) be a Lie group that acts smoothly on the manifolds \( M \) and \( N \).
	Assume that the \( G \)-action admits a slice \( S_m \subseteq M \) at a given point \( m \in M \) and that the \( G_m \)-action on \( N \) admits a slice \( S_n \) at the point \( n \in N \).
	Then, \( S_m \times S_n \) is a slice at \( (m, n) \) for the diagonal action of \( G \) on the product \( M \times N \).
\end{prop}
\begin{proof}
	The product \( S_m \times S_n \) is clearly a submanifold of \( M \times N \) containing the point \( (m, n) \).
	For \iref{i::slice:SliceDefSliceInvariantUnderStab} and \iref{i::slice:SliceDefOnlyStabNotMoveSlice}, note that
	\begin{equation}
		G_{(m, n)} = G_m \intersect G_n = (G_m)_n.
	\end{equation}
	Hence, using \iref{i::slice:SliceDefSliceInvariantUnderStab} for each slice, we see that \( S_m \times S_n \) is invariant under the stabilizer \( G_{(m, n)} \).
	Now let \( g \in G \), \( s_m \in S_m \) and \( s_n \in S_n \) be such that \( g \cdot (s_m, s_n) \in S_m \times S_n \).
	Then, \( g \cdot s_m \in S_m \) so that \iref{i::slice:SliceDefOnlyStabNotMoveSlice} applied to \( S_m \) implies \( g \in G_m \).
	Similarly, \( g \cdot s_n \in S_n \) and thus \( g \in (G_m)_n = G_{(m, n)} \).
	Next, we show that \iref{i::slice:SliceDefLocallyProduct} is fulfilled.
	For this purpose, let \( \chi_M: G \slash G_m \supseteq U_M \to G \) and \( \chi_N: G_m \slash (G_m)_n \supseteq U_N \to G_m \) be local sections, whose existence is ensured by \iref{i::slice:SliceDefLocallyProduct} (separately applied to each slice).
	The local diffeomorphisms \( G \isomorph U_M \times G_m \) and \( G_m \isomorph U_N \times (G_m)_n \) induced by these sections combine to a local diffeomorphism \( G \isomorph U_M \times U_N \times G_{(m, n)} \), which shows that \( G_{(m, n)} \) is a principal Lie subgroup of \( G \).
	In the remainder, we will suppress this local diffeomorphism in our notation and consider \( U_M \times U_N \) as an open neighborhood of the identity coset in \( G \slash G_{(m, n)} \).
	Under this identification, the map
	\begin{equation*}
		\chi: G \slash G_m \times G_m \slash (G_m)_n \supseteq U_M \times U_N \to G,
		\quad
		(\equivClass{g}, \equivClass{h}) \mapsto \chi_M(\equivClass{g}) \chi_N(\equivClass{h})
	\end{equation*}
	is a local section of \( G \to G \slash G_{(m, n)} \).
	It is now straightforward to see that 
	\begin{equation*}
		U_M \times U_N \times S_m \times S_n \to M \times N,
		\quad
		(\equivClass{g}, \equivClass{h}, s_m, s_n) \mapsto \chi_M(\equivClass{g}) \chi_N(\equivClass{h}) \cdot (s_m, s_n)
	\end{equation*}
	is a diffeomorphism onto its image and so \iref{i::slice:SliceDefLocallyProduct} holds.
	Indeed, the equation \( m' = \chi_M(\equivClass{g}) \chi_N(\equivClass{h}) \cdot s_m \) uniquely determines \( \equivClass{g} \) and \( \chi_N(\equivClass{h}) \cdot s_m \), because the latter is an element of \( S_m \) by \iref{i::slice:SliceDefSliceInvariantUnderStab}.
	Then, \( \equivClass{h} \) and \( s_n \) are uniquely determined by \( n' = \chi_M(\equivClass{g}) \chi_N(\equivClass{h}) \cdot s_n \), too.
	Finally, \iref{i:slice:linearSlice} is clearly satisfied, because the slices \( S_m \) and \( S_n \) are equivariantly diffeomorphic to an open subset in a certain linear representation space.
\end{proof}

\section{Orbit type stratification}
\label{sec:orbitTypeStrat}
In this section, we investigate the natural stratification of \( M \) and of the orbit space \( M / G \) by orbit types for a proper Lie group action of \( G \) on \( M \) admitting a slice at every point. 
We refer the reader to \cref{sec::stratification} for the notion of a stratified space and related background material.

\subsection{Orbit type stratification of \texorpdfstring{$M$}{M}}
\begin{prop} \label{prop::lieGroup:orbitTypeLocallyClosedLocalSubmanifold}
	Let \( M \) be a \( G \)-manifold admitting a slice at every point.
	For every stabilizer subgroup \( H \subseteq G \) the subset \( M_{(H)} \) of orbit type \( (H) \) is a locally closed submanifold of \( M \).
\end{prop}
\begin{proof}
	Since the statement is local, it is enough to consider a slice neighborhood at a given point \( m \in M_{(H)} \).
	Thus, let \( S \) be a slice at \( m \). 
	Choose a local section \( \chi: G/G_m \supseteq U \to G \) defined in an open neighborhood \( U \) of the identity coset in such a way that the map
	\begin{equation}
		\chi^S: U \times S \to M, \qquad (\equivClass{g}, s) \mapsto \chi(\equivClass{g}) \cdot s
	\end{equation}
	is a diffeomorphism onto an open neighborhood \( V \subseteq M \) of \( m \), see \iref{i::slice:SliceDefLocallyProduct}.
	Since points on the same \( G \)-orbit have conjugate stabilizers, a point \( v = \chi(\equivClass{g}) \cdot s \in V \) lies in \( M_{(H)} \) if and only if the corresponding point \( s \in S \) has orbit type \( (H) \).
	Therefore, \( \chi^S \) maps \( U \times S_{(H)} \) diffeomorphically onto \( V_{(H)} = V \intersect M_{(H)} \).
	Now the claim follows, because the partial slice \( S_{(H)} = S_{(G_m)} \) is a closed submanifold of \( S \) according to \iref{i::slice:SliceDefPartialSliceSubmanifold}, and because the slice itself is closed in \( V \).
\end{proof}

\begin{thm}[Orbit Type Stratification] \label{prop::stratification:orbitTypePartionIsStratification}
	Let \( M \) be a \( G \)-manifold with proper \( G \)-action admitting a slice at every point.
	Assume that the approximation property \( M_{\geq (H)} \subseteq \closureSet{M_{(H)}} \) holds for every stabilizer subgroup \( H \subseteq G \).
	Then, the partition of \( M \) into orbit type submanifolds \( M_{(H)} \) is a stratification.
\end{thm}
We call the stratification of \( M \) into the orbit type subsets \( M_{(H)} \) the \emphDef{stratification by orbit types}.
\begin{proof}
	By assumption, \( M \) is a Hausdorff space and, by \cref{prop::lieGroup:orbitTypeLocallyClosedLocalSubmanifold}, the orbit type subsets \( M_{(H)} \) are locally closed submanifolds.
	We will use \cref{prop::stratification:ClosureFormulaImpliesFrontierCondition} to show that the partition into orbit type submanifolds is a stratification.
	Since the action is proper, the set of orbit types is partially ordered by \cref{prop::properAction:orbitTypeIsPartialOrder}.
	Thus it remains to show that \( \closureSet{M_{(H)}} = M_{\geq (H)} \) holds for every stabilizer subgroup \( H \subseteq G \).
	By assumption, \( M_{\geq (H)} \subseteq \closureSet{M_{(H)}} \).
	The existence of slices implies the converse inclusion.
	Indeed, for every \( m \in \closureSet{M_{(H)}} \), the open slice neighborhood \( V \) at \( m \) has a non-empty intersection with \( M_{(H)} \).
	However, \cref{prop::slice:sliceNeighbourhoodHasSubconjugatedStabilizer} entails that \( V \subseteq M_{\leq (G_m)} \).
	Hence, \( (G_m) \geq (H) \) and, therefore, \( m \in M_{\geq (H)} \).
	In summary, \cref{prop::stratification:ClosureFormulaImpliesFrontierCondition} implies that the partition into orbit type submanifolds is a stratification.
	Moreover, the natural partial order of strata is opposite to the partial order of conjugacy classes.
	That is, \( M_{(K)} \subseteq \closureSet{M_{(H)}} \) (or equivalently, \( M_{(K)} \leq M_{(H)} \) in the notation of \cref{def:stratification}) if and only if \( (K) \geq (H) \). 
\end{proof}
The approximation property \( M_{\geq (H)} \subseteq \closureSet{M_{(H)}} \) does not always hold, not even in finite dimensions, as the example below shows.
At least in finite dimensions, it is always possible to find a refinement into connected components such that a stratification is obtained.
Indeed, a slice analysis shows that the set \( M_{(K)} \intersect \closureSet{M_{(H)}} \) is open and closed in \( M_{(K)} \) for every \( (K) \geq (H) \), see \parencite[Proposition~4.3.2.3]{Pflaum2001a}.
We do not know, whether a similar refinement into connected components is possible in infinite dimensions.
\begin{example}
	Consider the action of \( \UGroup(1) \times \UGroup(1) \) on \( \CP^2 \) given in homogeneous coordinates by
	\begin{equation}
		(\E^{\I \phi_1}, \E^{\I \phi_2}) \cdot (z_1 : z_2 : z_3) \defeq (\E^{\I \phi_1} z_1 : \E^{\I \phi_2} z_2 : z_3).
	\end{equation}
	The non-trivial stabilizers and their corresponding orbit type strata are listed in \cref{tab:u1u1actioncp2:strata}.
	\begin{table}	
		\centering
		\begin{tabular}{lcc}
			\toprule
			Stabilizer \( H \)		
				& \( M_H = M_{(H)} \)	
				& \( \closureSet{M_{(H)}} \setminus M_{(H)} \)
				\\  \midrule
			\( \diag \UGroup(1) \)
				& \( \set{(z_1 \neq 0 : z_2 \neq 0 : 0)} \)
				& \( \set{(1 : 0 : 0), (0 : 1 : 0)} \)
				\\
			\( \set{e} \times \UGroup(1) \)
				& \( \set{(z_1 \neq 0 : 0 : z_3 \neq 0)} \)
				& \( \set{(1 : 0 : 0), (0 : 0 : 1)} \)
				\\
			\( \UGroup(1) \times \set{e} \)
				& \( \set{(0 : z_2 \neq 0 : z_3 \neq 0)} \)
				& \( \set{(0 : 1 : 0), (0 : 0 : 1)} \)
				\\
			\( \UGroup(1) \times \UGroup(1) \)
				& \( \set{(1 : 0 : 0), (0 : 1 : 0), (0 : 0 : 1)}  \)
				&
				\\  \bottomrule	
		\end{tabular}
		\caption{Orbit type strata of the \( \UGroup(1) \times \UGroup(1) \)-action on \( \CP^2 \).}
		\label{tab:u1u1actioncp2:strata}
	\end{table}
	None of the intermediate orbit type strata contains all points with stabilizer \( \UGroup(1) \times \UGroup(1) \) in its closure and so the frontier condition is violated.
	Note, however, that \( M_{(H)} \intersect \closureSet{M_{(K)}} \neq \emptyset \) is still equivalent to \( (H) \geq (K) \).
	After decomposing the fixed point set into its connected components one obtains a partition for which the frontier condition is satisfied.
	The refined partial order can be visualized in the following Hasse diagram.
	An arrow from one stratum to another signifies that the target lies in the closure of the source.
	\begin{center}
		\begin{tikzpicture}[
			level/.style={sibling distance=10.5em/#1},
			edge from parent/.style={draw, ->}]
			\node (z){regular} [grow=up]
				child {node (diag) {$\set{(z_1 \neq 0 \tikzcolon z_2 \neq 0 \tikzcolon 0)}$}
					child {node (100) {$\set{(1 \tikzcolon 0 \tikzcolon 0)}$}}
				}
				child {node (eU) {$\set{(z_1 \neq 0 \tikzcolon 0 \tikzcolon z_3 \neq 0)}$}
					child {node (010) {$\set{(0 \tikzcolon 1 \tikzcolon 0)}$}
						edge from parent[draw=none]
					}
				}
				child {node (Ue) {$\set{(0 \tikzcolon z_2 \neq 0 \tikzcolon z_3 \neq 0)}$}
					child {node (001) {$\set{(0 \tikzcolon 0 \tikzcolon 1)}$}}
				}
				;
			\draw[->] (diag) -- (010);
			\draw[->] (eU) -- (100);
			\draw[->] (eU) -- (001);
			\draw[->] (Ue) -- (010);
		\end{tikzpicture}
	\end{center}
	See \parencites[Example~B]{Ferrarotti1994}[Example~17]{CrainicMestre2017} for other examples where the orbit type decomposition does not satisfy the frontier condition. 
\end{example}

Besides the orbit type manifold \( M_{(H)} \) also the set \( M_H \) of points with stabilizer \( H \) is a submanifold of \( M \), see \eg \parencite[Corollary~4.2.8]{Pflaum2001a} for the analogous statement in finite dimensions.
In infinite dimensions, we need to assume that the normalizer is a Lie subgroup, which is no longer an automatic consequence of closedness.
\begin{prop}
	\label{prop::lieGroup:isotropyTypeSubmanifold}
	Let \( M \) be a \( G \)-manifold with proper \( G \)-action admitting a slice at every point.
	Let \( H \subseteq G \) be a stabilizer subgroup.
	If the normalizer \( \normalizer_G (H) \) of \( H \) in \( G \) is a Lie subgroup of \( G \), then the set \( M_{H} \) of isotropy type \( H \) is a submanifold of \( M_{(H)} \).
	Moreover, at \( m \in M_H \), we have
	\begin{equation}
		\TBundle_m M_H \subseteq (\TBundle_m M)_H,
	\end{equation}
	where \( H \) acts on \( \TBundle_m M \) via the isotropy representation\footnote{The element \( h \in H \) acts on a tangent vector \( X \in \TBundle_m M \) represented by a curve \( \gamma: [0,1] \to M \) via \( h \ldot X \defeq \difFracAt{}{\varepsilon}{0} h \cdot \gamma(\varepsilon) \in \TBundle_m M \).}.
\end{prop}
\begin{proof}
	Let \( m \in M_H \) and let a slice \( S \) at \( m \) be chosen.
	Recall from the proof of \cref{prop::lieGroup:orbitTypeLocallyClosedLocalSubmanifold} that the slice diffeomorphism \( \chi^S \) locally identifies \( M_{(H)} \) with \( U \times S_{(H)} \), where \( U \) is an open neighborhood of the identity coset \( \equivClass{e} \) in \( G \slash H \).
	We claim that under \( (\chi^S)^{-1} \) the isotropy subset \( M_H \) takes the form \( \normalizer_{G \slash H} (U) \times S_H \), where
	\begin{equation}
		\normalizer_{G \slash H} (U) \defeq \set*{\equivClass{g} \in U \subseteq G \slash H \given \chi(\equivClass{g}) \in \normalizer_G (H)}.
	\end{equation}
	Indeed, a point \( v = \chi(\equivClass{g}) \cdot s \) in the slice neighborhood \( V \subseteq M \) has stabilizer \( H \) only if \( G_s \) is conjugate to \( H \).
	However, \( G_s \subseteq H \) by \cref{prop::slice:mHasMaximalStabilizerOfWholeSlice}.
	Thus, \cref{prop::compactLieSubgroup:conjugatedSubgroupEqual} implies \( G_s = H \).
	Moreover, we have \( H = G_v = \chi(\equivClass{g}) H \chi(\equivClass{g})^{-1} \) and so \( \chi(\equivClass{g}) \in \normalizer_G (H) \).
	Conversely, if \( G_s = H \) and \( \chi(\equivClass{g}) \in \normalizer_G (H) \), then the point \( v = \chi(\equivClass{g}) \cdot s \) has stabilizer \( H \).
	Since, moreover, \( S_{(H)} = S_H \) by \cref{prop:properAction:sliceOrbitTypeEqualsSliceStab}, we find
	\begin{equation}
		\normalizer_{G \slash H} (U) \times S_H \isomorph M_H \subseteq M_{(H)} \isomorph U \times S_{H}.
	\end{equation}
	To conclude that \( M_H \) is a submanifold of \( M_{(H)} \) it thus suffices to show that \( \normalizer_{G \slash H} (U) \) is a submanifold of \( U \).
	For that purpose, recall that the section \( \chi \) induces a local trivialization of \( G \to G \slash H \) via the map
	\begin{equation}
		U \times H \to G, \quad (\equivClass{g}, h) \mapsto \chi(\equivClass{g}) h,
	\end{equation}
	which is a diffeomorphism onto its image.
	Under this diffeomorphism, the Lie subgroup \( \normalizer_G (H) \) is mapped to \( \normalizer_{G \slash H} (U) \times H \) as a straightforward calculation reveals.
	In particular, \( \normalizer_{G \slash H} (U) \) is a submanifold of \( U \).

	In order to show \( \TBundle_m M_H \subseteq (\TBundle_m M)_H \), let \( X \in \TBundle_m M_H \) be a tangent vector represented by a curve \( \gamma: [0,1] \to M_H \).
	For every \( h \in H \), we have
	\begin{equation}
		h \ldot X = \difFracAt{}{\varepsilon}{0} h \cdot \gamma(\varepsilon) = \difFracAt{}{\varepsilon}{0} \gamma(\varepsilon) = X,
	\end{equation}
	because \( \gamma(\varepsilon) \in M_H \).
	Thus, \( \TBundle_m M_H \subseteq (\TBundle_m M)_H \), indeed.
\end{proof}

\subsection{Orbit type stratification of \texorpdfstring{$M / G$}{M / G}}
We also get an analogous stratification of the orbit space \( M \slash G \).
\begin{prop} \label{prop::lieGroup:orbitTypeQuotientManifold}
	Let \( M \) be a \( G \)-manifold with proper \( G \)-action admitting a slice at every point.
	For every stabilizer subgroup \( H \subseteq G \), the quotient \( \check{M}_{(H)} = M_{(H)} \slash G \) carries a unique manifold structure such that the canonical projection \( \pi_{(H)}: M_{(H)} \to \check{M}_{(H)} \) is a smooth submersion.
\end{prop}
\begin{proof}
	Since \( M_{(H)} \) is \( G \)-invariant, the \( G \)-action on \( M \) restricts to a proper \( G \)-action on \( M_{(H)} \).
	Thus, by \cref{prop:properAction:orbitSpaceHausdorff}, the quotient \( \check{M}_{(H)} = M_{(H)} \slash G \) endowed with the quotient topology is a Hausdorff space.
	Let \( \pi_{(H)}: M_{(H)} \to \check{M}_{(H)} \) denote the canonical projection. 

	First, let us construct local charts on \( \check{M}_{(H)} \).
	Let \( m \in M_{(H)} \) and let \( S \) denote a slice at \( m \).
	Let \( \chi: G/G_m \supseteq U \to G \) be a local section such that the associated map \( \chi^S: U \times S \to M \) is an diffeomorphism onto an open slice neighborhood \( V \) of \( m \), as in \iref{i::slice:SliceDefLocallyProduct}.
	Then, \( V_{(H)} = V \intersect M_{(H)} \) is open in \( M_{(H)} \) and, since \( \pi_{(H)} \) is an open map, \( \pi_{(H)}(V_{(H)}) = V_{(H)} \slash G \) is an open neighborhood of \( G \cdot m \) in \( \check{M}_{(H)} \).
	Recall from the proof of \cref{prop::lieGroup:orbitTypeLocallyClosedLocalSubmanifold} that \( \chi^S \) restricts to a diffeomorphism \( \chi^S_{(H)} \) from \( U \times S_{(H)} \) onto \( V_{(H)} \).
	Consider the composition
	\begin{equation}
		\pi_{(H)} \circ \chi^S_{(H)}: U \times S_{(H)} \to \check{M}_{(H)},
	\end{equation}
	which is obviously a surjection onto \( \check{V}_{(H)} \equiv V_{(H)} \slash G \).
	We claim that this map is also injective when restricted to \( \equivClass{e} \times S_{(H)} \).
	Indeed, if two points \( s \) and \( \tilde{s} \) of \( S_{(H)} \) are mapped to the same point in \( \check{M}_{(H)} \), then there exists \( a \in G \) such that \( a \cdot s = \tilde{s} \).
	Hence, \( a \) is an element of the stabilizer \( G_m \) due to \iref{i::slice:SliceDefOnlyStabNotMoveSlice}. 
	Since, by assumption, \( s \in S_{(H)} \) and since \( S_{(H)} = S_{G_m} \) according to \cref{prop:properAction:sliceOrbitTypeEqualsSliceStab}, we get \( s = \tilde{s} \).
	In summary, \( \pi_{(H)} \circ \chi^S_{(H)} \) restricts to a continuous bijection \( \varphi_m: S_{(H)} \to \check{V}_{(H)} \).
	Moreover, \( \varphi_m \) is open and hence a homeomorphism.
	
	We wish to put a smooth structure on \( \check{M}_{(H)} \) by requiring that \( \varphi_m \) be a diffeomorphism for all \( m \in M_{(H)} \).
	In order to make sense of such a procedure, we need to verify that the notion of smoothness does not depend on the chosen point \( m \) in the orbit \( G \cdot m \).
	Let \( m' = a \cdot m \), with \( a \in G \), be another point on the orbit \( G \cdot m \).
	Choose a slice \( S' \) at \( m' \) with slice neighborhood \( V' \subseteq M \) and with associated diffeomorphism \( \chi'^{S'}_{(H)}: U' \times S'_{(H)} \to V'_{(H)} \).
	By shrinking \( S \), we can assume that \( a \cdot S \subseteq V' \).
	Then, the transition map between the two charts \( \varphi_m \) and \( \varphi_{m'} \) is given by
	\begin{equation}
		\varphi_{m'}^{-1} \circ \varphi_m (s) = \pr_{S'_{(H)}} \circ (\chi'^{S'}_{(H)})^{-1} (a \cdot s).
	\end{equation}
	Hence, it is a smooth map.
	In summary, the charts \( \varphi_m: S_{(H)} \to \check{V}_{(H)} \) endow \( \check{M}_{(H)} \) with a smooth structure.
	With respect to these charts, the map \( \pi_{(H)} \) is represented by the projection \( U \times S_{(H)} \to S_{(H)} \) onto the second factor and so it is a smooth submersion.
\end{proof}

\begin{thm}[Orbit Type Stratification of the Orbit Space]
	\label{prop:orbitTypeStratification:ofOrbitSpace}
	Let \( M \) be a \( G \)-manifold with proper \( G \)-action admitting a slice at every point.
	Assume that the approximation property \( M_{\geq (H)} \subseteq \closureSet{M_{(H)}} \) holds for every stabilizer subgroup \( H \subseteq G \).
	Then, the partition of the orbit space \( \check{M} = M \slash G \) into the subsets \( \check{M}_{(H)} = M_{(H)} \slash G \) is a stratification.
\end{thm}
We call the stratification of \( \check{M} \) into the orbit type subsets \( \check{M}_{(H)} \) the \emphDef{stratification by orbit types}.
\begin{proof}
	Since the \( G \)-action is proper, the quotient topology on the orbit space \( \check{M} \) is Hausdorff according to \cref{prop:properAction:orbitSpaceHausdorff}. 
	By \cref{prop::lieGroup:orbitTypeQuotientManifold}, \( \check{M}_{(H)} \) is a smooth manifold for every orbit type \( (H) \).
	Moreover, \( \check{M}_{(H)} \) is locally closed in \( \check{M} \), because every slice \( S \) is closed in \( G \cdot S \) due to \cref{prop::slice:SliceDefSliceSweepIsOpen} and because \iref{i::slice:SliceDefPartialSliceSubmanifold} entails that the partial slice \( S_{H} \) is closed in \( S \).
	Hence, the property \iref{i::stratification:stratumIsManifold} of \cref{def:stratification} is satisfied.

	The frontier property \iref{i::stratification:frontierCondition} of the partition of \( \check{M} \), see \cref{def:stratification}, is derived from that of the orbit type stratification of \( M \).
	Indeed, let \( (H) \neq (K) \) be orbit types such that \( \check{M}_{(H)} \intersect \closureSet{\check{M}_{(K)}} \) is not empty.
	We have
	\begin{equation}\label{eq:orbitTypeStratification:relationFrontierOrbitSpace}\begin{split}
		\check{M}_{(H)} \intersect \closureSet{\check{M}_{(K)}} 
			&= \pi\left(M_{(H)}\right) \intersect \closureSet{\pi\left(M_{(K)}\right)} \\
			&= \pi\left(M_{(H)}\right) \intersect \pi\bigl(\closureSet{M_{(K)}}\bigr) \\
			&= \pi\bigl(M_{(H)} \intersect \closureSet{M_{(K)}}\bigr),
	\end{split}\end{equation}
	where the non-trivial inclusions in this chain follow from \( G \)-invariance of \( M_{(H)} \) and \( M_{(K)} \).
	Hence, \( M_{(H)} \intersect \closureSet{M_{(K)}} \neq \emptyset \).
	Since the decomposition of \( M \) into orbit types satisfies the frontier condition, we have \( M_{(H)} \subseteq \closureSet{M_{(K)}} \setminus M_{(K)} \) and \( M_{(K)} \intersect \closureSet{M_{(H)}} = \emptyset \).
	Therefore, by similar arguments as above, \( \check{M}_{(H)} \subseteq \pi\bigl(\closureSet{M_{(K)}} \setminus M_{(K)}\bigr) = \closureSet{\check{M}_{(K)}} \setminus \check{M}_{(K)} \) and \( \check{M}_{(K)} \intersect \closureSet{\check{M}_{(H)}} = \emptyset \), which verifies the frontier condition for the orbit type decomposition of \( \check{M} \).
\end{proof}

\subsection{Orbit type stratification of a subset}
Often one is rather interested in a subset \( X \subseteq M \) than in the whole space \( M \).
For example, in the context of symplectic reduction, \( M \) is the phase space and \( X \) is a momentum map level set.
Under natural conditions on the behavior of the slices relative to \( X \), the subset \( X \) inherits the orbit type stratification from \( M \).
\begin{prop}
	Let \( M \) be a \( G \)-manifold with proper \( G \)-action and let \( X \) be a closed \( G \)-invariant subset of \( M \).
	Assume the \( G \)-action on \( M \) admits a slice \( S \) at every point \( m \in X \) such that
	\begin{enumerate}
		\item 
			\( X \intersect S_{(G_m)} \) is a closed submanifold of \( S_{(G_m)} \),
		\item
			for every orbit type \( (K) \leq (G_m) \), the point \( m \) lies in the closure of \( S_{(K)} \intersect X \) in \( S \).
	\end{enumerate}
	Then, the induced partition of \( X \) into the orbit type subsets \( X_{(H)} = X \intersect M_{(H)} \) is a stratification.
	Moreover, under these assumptions, the decomposition of \( \check{X} = X \slash G \) into \( \check{X}_{(H)} = X_{(H)} \slash G \) is a stratification, too.
\end{prop}
\begin{proof}
	We will again employ \cref{prop::stratification:ClosureFormulaImpliesFrontierCondition} to show that the partition of \( X \) into orbit type subsets is a stratification.
	Let \( (H) \) be an orbit type and let \( m \in X_{(H)} \).
	Choose a slice \( S \) at \( m \) having the properties stated above.
	Let \( V \subseteq M \) denote the slice neighborhood of \( m \) as in \iref{i::slice:SliceDefLocallyProduct}.
	As we have seen in the proof of \cref{prop::lieGroup:orbitTypeLocallyClosedLocalSubmanifold}, the slice diffeomorphism \( \chi^S: U \times S \to V \) locally identifies the orbit type manifold \( M_{(H)} \) with \( U \times S_{(G_m)} \).
	Since \( X \) is \( G \)-invariant, the subset \( X_{(H)} = X \intersect M_{(H)} \) is identified with \( U \times (X \intersect S_{(G_m)}) \).
	By assumption, \( X \intersect S_{(G_m)} \) is a closed submanifold of \( S_{(G_m)} \) and thus \( U \times (X \intersect S_{(G_m)}) \) is a closed submanifold of \( U \times S_{(G_m)} \).
	Accordingly, \( X_{(H)} \) is a locally closed submanifold of \( M_{(H)} \).
	Given that \( X_{(H)} \intersect S = X \intersect S_{(G_m)} \) is closed in \( X \intersect S \), the piece \( X_{(H)} \) is locally closed in \( X \).

	The set of orbit types occurring in \( X \) is, of course, a subset of the set of all orbit types of \( M \) and hence it inherits the partial order.
	For the stratification of \( X \), it remains to show that \( \closureSet{X}_{(K)} = \bigUnion_{(H) \geq (K)} X_{(H)} \) holds for every orbit type \( (K) \) of \( X \).
	One inclusion is immediate from
	\begin{equation}
		\closureSet{X}_{(K)} = \closureSet{X \intersect M_{(K)}} 
			\subseteq X \intersect \closureSet{M_{(K)}}
			\subseteq \bigUnion_{(H) \geq (K)} X \intersect M_{(H)},
	\end{equation}
	where we have used the relation \( \closureSet{M_{(K)}} \subseteq M_{\geq (K)} \) which follows from the existence of slices for the \( G \)-action on \( M \), see the proof of \cref{prop::stratification:orbitTypePartionIsStratification}.
	For the converse inclusion, choose a point \( m \in X_{(H)} \) with \( (H) \geq (K) \).
	In slice coordinates at \( m \), we have \( V \intersect X_{(K)} \isomorph U \times (X \intersect S_{(K)}) \).
	By assumption, the point \( m \) lies in the closure of \( X \intersect S_{(K)} \) and so it also lies in \( \closureSet{X}_{(K)} \).
	Thus, \( X_{\geq (K)} \subseteq \closureSet{X}_{(K)} \).
	Now, \cref{prop::stratification:ClosureFormulaImpliesFrontierCondition} shows that the decomposition of \( X \) into \( X_{(H)} \) is a stratification.

	The proof that the decomposition of \( \check{X} \) into the sets \( \check{X}_{(H)} \) is a stratification follows the steps in the proofs of \cref{prop::lieGroup:orbitTypeQuotientManifold,prop:orbitTypeStratification:ofOrbitSpace} and, thus, will only be sketched here.
	Similarly to the situation above, the composition
	\begin{equation}
		\pi_{(H)} \circ \chi^S_{(H)}: U \times (X \intersect S_{(H)}) \to \check{X}_{(H)}
	\end{equation}
	yields a homeomorphism \( \varphi_m: X \intersect S_{(H)} \to X \intersect \check{V}_{(H)} \), which we take as a chart on \( \check{X}_{(H)} \).
	The frontier property of the decomposition of \( \check{X} \) is inherited from the one of \( X \).
	The reasoning is identical to the proof of \cref{prop:orbitTypeStratification:ofOrbitSpace}.
	However, this time the calculation~\eqref{eq:orbitTypeStratification:relationFrontierOrbitSpace} has to be replaced by
	\begin{equation}\begin{split}
		\check{X}_{(H)} \intersect \closureSet{\check{X}_{(K)}} 
			&= \pi\left(X_{(H)}\right) \intersect \closureSet{\pi\left(X_{(K)}\right)} \\
			&= \pi\left(X_{(H)}\right) \intersect \pi\left(\closureSet{X_{(K)}}\right) \\
			&= \pi\left(X_{(H)} \intersect \closureSet{X_{(K)}}\right),
	\end{split}\end{equation}
	which is justified again by \( G \)-invariance of the sets \( X_{(H)} \) and \( X_{(K)} \).
\end{proof}
	
\subsection{Local finiteness}
\label{sec:orbitTypeStratification:locallyFinite}
In the finite-dimensional setting, it is often helpful to know that the stratification is locally finite.
This feature is crucial for understanding the local structure, for instance to construct or analyze a neighborhood of a point by recursive methods.
In order to show that the orbit type stratification is locally finite, a linear slice can be used to reduce the problem to the linear action of a compact group on a vector space.
For finite-dimensional vector spaces, every linear action by a compact group has only finitely many orbit types \parencite[Proposition~4.2.8]{Sniatycki2013}.
However, the proof uses induction on the dimension of the vector space and thus does not generalize to infinite-dimensional spaces.
In fact, we now give an example of an action of a compact group on an infinite-dimensional vector space with infinitely many orbit types.
\begin{example}
	\label{ex:orbitTypeStratification:locallyFinite:testFunctions}
	Let \( M \) be a finite-dimensional manifold and let \( G \) be a compact Lie group which acts smoothly on \( M \).
	The vector space \( \scFunctionSpace(M) \) of test functions has as the topological dual the space of distributions \( \DistributionSpace(M) \).
	The \( G \)-action on \( M \) yields an action on test functions by precomposition and thus also gives a linear action on \( \DistributionSpace(M) \):
	\begin{equation}
		\dualPair{g \cdot T}{\phi} = \dualPair{T}{\phi \circ g} \quad \text{for } T \in \DistributionSpace(M), \phi \in \scFunctionSpace(M).
	\end{equation}
	Now, for the delta distribution \( T = \delta_m \) at a point \( m \in M \) we have
	\begin{equation}
		\dualPair{g \cdot \delta_m}{\phi} = \dualPair{\delta_m}{\phi \circ g} = \phi(g \cdot m) = \dualPair{\delta_{g \cdot m}}{\phi}.
	\end{equation}
	Hence, \( g \cdot \delta_m = \delta_{g \cdot m} \).
	In particular, the stabilizer of \( \delta_m \) is equal to \( G_m \).
	Thus, there are at least as many orbit types for the action of \( G \) on \( \DistributionSpace(M) \) as there are orbit types for the \( G \)-action on \( M \).
	However, there are examples of actions of compact Lie groups on non-compact manifolds with infinitely many orbit types\footnote{
		For example, consider an ensemble of countably many disks.
		Since each disk may be rotated independently with a different speed, we can construct an action of \( \UGroup(1) \) such that all the subgroups \( \Z \slash n \Z \) with \( n \in \N \) occur as stabilizers.
		The disks may be glued together to yield an example that does not rely on having infinitely many connected components, see \parencite{Yang1957} for details. 
	}.
	Therefore, in such cases, the action on \( \DistributionSpace(M) \) has infinitely many orbit types.
\end{example}

\subsection{Regularity properties}

It is well-known that the orbit type stratification of a finite-dimensional \( G \)-manifold satisfies the so-called Whitney regularity conditions (A) and (B), see \parencite[Definition~1.4.3]{Pflaum2001a}.
The quintessence of these regularity conditions is to control how the tangent space to the stratum behaves when one moves closer to the frontier.
The Whitney condition (A) entails that for any sequence of points \( x_i \) in a stratum \( X_\sigma \) that converges to \( x \in X_\varsigma \) and for which the sequence of tangent spaces \( \TBundle_{x_i} X_\sigma \) converges in the Grassmann bundle of \( \TBundle X \) to a space \( \tau \), we have \( \TBundle_x X_\varsigma \subseteq \tau \).
The Whitney condition (B) is similar in spirit and also involves a limit of tangent spaces.
A more quantitative statement is required by the Verdier condition, which roughly states that for nearby points \( x \in X_\sigma \) and \( y \in X_\varsigma \) there exists a constant \( C > 0 \) such that
\begin{equation}
	d_{\textrm{Gr}} \left(\TBundle_x X_\sigma \, , \TBundle_y X_\varsigma\right) < C \, \norm{x - y},
\end{equation}
where \( d_{\textrm{Gr}} \) is a certain metric on the Grassmann bundle.
All these regularity conditions involve a limit procedure or a metric on the Grassmann bundle.
Although the Grassmann bundle clearly is the right arena to talk about the relation between different tangent spaces, it is not straightforward to generalize these notions of regularity to the infinite-dimensional setting. 
In particular, in infinite dimensions, the theory of continuous linear maps is more delicate and there is no natural topology on the Grassmann bundle.
Thus, it is not clear how the Whitney or Verdier conditions can be transferred to the infinite-dimensional context.

In the finite-dimensional context, \textcite{Giacomoni2014} observed that the orbit type stratification of \( M \) satisfies a very strong form of the Verdier condition.
Namely, \( d_{\textrm{Gr}} \left(\TBundle_x X_\sigma \,, \TBundle_y X_\varsigma\right) = 0 \), because the tangent space to a frontier stratum is contained in the tangent space of the original stratum.
This property does admit a generalization to the infinite-dimensional setting.
In order to formalize the observation of \parencite{Giacomoni2014}, we introduce the following regularity condition.
For the following, the reader might wish to recall \cref{def:stratification}.
\begin{defn}
	Let \( X \) be a stratified space.
	We say that a pair \( (X_\varsigma, X_\sigma) \) of strata satisfying \( X_\varsigma < X_\sigma \) is \emphDef{smoothly regular} if for all \( x \in X_\varsigma \) there exists an open neighborhood \( U \) of \( x \) in \( X \) such that for all \( y \in U \intersect X_\sigma \) there exists an injective continuous linear map \( \TBundle_x X_\varsigma \to \TBundle_y X_\sigma \).
	A stratified space is called \emphDef{smoothly regular} if all its strata are pairwise smoothly regular.
\end{defn}
\begin{figure}
	\centering
	\begin{tikzpicture}
		\fill[halfgray, opacity=0.2] (0, 0) rectangle (6,2.5);
		\draw[thick] (0,0) -- node[pos=0.95, below right]{$X_\varsigma$} (6,0);
		\path (5.5, 1.5) node[above] {$X_\sigma$};
		
		\filldraw[fill=blue2, opacity=0.3] (4.5, 0) arc [start angle=0, end angle=180, radius=2];
		\draw[blue2, thick] (0.5, 0) -- (4.5, 0);
		\path (4, 0.2) node[above, blue2] {$U$};

		\fill (2.5, 0) circle [radius=1.5pt] node(x){};
		\fill (1.5, 1) circle [radius=1.5pt] node(y){};
		\path (x) node[below] {$x$};
		\path (y) node[below] {$y$};

		\draw[-latex, thick] (x.center) -- node[below right] {$v$} +(0.8, 0); 
		\draw[-latex, thick] (y.center) -- node[below right] {$v$} +(0.8, 0); 
	\end{tikzpicture}
	\caption{Illustration of a pair \( (X_\varsigma, X_\sigma) \) of smoothly regular strata. In this case, the injection \( \TBundle_x X_\varsigma \to \TBundle_y X_\sigma \) is given by translating the vector \( v \in \TBundle_x X \) to the point \( y \).}
	\label{fig:orbitTypeStratification:smoothlyRegular}
\end{figure}
Informally speaking, the tangent spaces to the strata in a smoothly regular stratification locally only get bigger, see \cref{fig:orbitTypeStratification:smoothlyRegular} for a graphical illustration.
Since we have already seen in \cref{prop::slice:sliceNeighbourhoodHasSubconjugatedStabilizer} that (locally) the stabilizer only decreases, and since the size of the stabilizer is roughly inversely proportional to the size of the corresponding orbit, one would expect that the orbit type stratification is smoothly regular.
In fact, the proof of \parencite{Giacomoni2014} generalizes to the infinite-dimensional realm with only a few slight modifications.
\begin{prop}
	Let \( M \) be a \( G \)-manifold with proper \( G \)-action admitting a linear slice \( S \) at every point.
	Assume that the \( G_m \)-action on \( S \) has a linear slice\footnote{By \cref{prop:slice:sliceTheoremLinearAction}, this assumption is redundant for a Fréchet \( G \)-manifold \( M \) with proper \( G \)-action.}.
	Moreover, assume that the approximation property \( M_{\geq (H)} \subseteq \closureSet{M_{(H)}} \) holds for every orbit type \( (H) \).
	Let \( (K) \leq (H) \) be two isotropy types and let \( m \in M_{(H)} \).
	Then, there exists an open neighborhood \( V \) of \( m \) in \( M \) such that for all \( v \in V \intersect M_{(K)} \) we have an injective continuous linear map \( \TBundle_m M_{(H)} \to \TBundle_v M_{(K)} \).
	In particular, the orbit type stratification of \( M \) is smoothly regular.
\end{prop}
\begin{proof}
	As this is a local statement, we can investigate it in a slice neighborhood.
	Let \( S \) be a linear slice at \( m \).
	Without loss of generality, we may regard \( S \) as an open neighborhood of \( 0 \) in a locally convex vector space \( X \) carrying a linear \( G_m \)-representation, see \cref{def:slice:linearSlice}.
	We identify an open neighborhood \( V \) of \( m \) with \( U \times S \) via the slice diffeomorphism \( \chi^S: U \times S \to V \) as in \iref{i::slice:SliceDefLocallyProduct}.
	By equivariance of the stabilizer subgroup, we have \( V_{(H)} \isomorph U \times S_{(H)} \) and \( V_{(K)} \isomorph U \times S_{(K)} \).
	Since the \( G_m \) action on \( S \) has slices, \cref{prop::lieGroup:orbitTypeLocallyClosedLocalSubmanifold} shows that \( S_{(H)} \) and \( S_{(K)} \) are submanifolds of \( S \).
	Thus, viewing the tangent spaces of \( S_{(H)} \) and of \( S_{(K)} \) as subspaces of \( X \), it suffices to show that
	\begin{equation}
		\label{eq:regularityProperties:sliceOrbitTypeInclusion}
		\TBundle_m S_{(H)} \subseteq \TBundle_s S_{(K)}	
	\end{equation}
	holds for all \( s \in S_{(K)} \) sufficiently close to \( m \).
	
	Next, choose a representative \( K \) of \( (K) \) such that \( K \subseteq G_m \).
	Recall from \cref{prop:properAction:sliceOrbitTypeEqualsSliceStab} that \( S_{(H)} = S_{G_m} \), because the action is proper.
	Thus, it is enough to prove that \( \TBundle_m S_{G_m} \subseteq \TBundle_s S_{K} \) as \( S_K \) is a submanifold of \( S_{(K)} \) by \cref{prop::lieGroup:isotropyTypeSubmanifold}.
	For that purpose, first note that the subset \( X_{\geq K} \) of \( K \)-invariant vectors is a closed linear subspace of \( X \).
	Indeed, \( X_{\geq K} \) is a linear subspace, because the action is linear, and closedness follows from the representation
	\begin{equation}
		X_{\geq K} = \bigIntersection_{k \in K} X_k,
	\end{equation}
	where \( X_k \defeq \set{x \in X \given k \cdot x = x} \) is closed for all \( k \in X \).
	Since the action of \( G_m \) on \( X \) admits slices, \( S_K \) is open in \( X_{\geq K} \) according to \cref{prop:properAction:sliceOrbitTypeOpenInSupOrbitType} and so \( \TBundle_s S_K \isomorph X_{\geq K} \).
	But, of course, \( X_{G_m} \isomorph \TBundle_m S_{G_m} \) is contained in \( X_{\geq K} \) as \( H \supseteq K \).
	Thus, \( \TBundle_m S_{G_m} \subseteq \TBundle_s S_K \).
\end{proof}

\appendix
\section{Stratified Spaces}
\label{sec::stratification}

Roughly speaking, a stratified space is a topological space that is built from a collection of manifolds fitting together in a particularly nice way.
\begin{defn}
	\label{def:stratification}
	Let \( X \) be Hausdorff topological space. 
	A partition \( \stratification{Z} \) of \( X \) into subsets \( X_\sigma \) indexed by \( \sigma \in \Sigma \) is called a \emphDef{stratification} of \( X \) if the following conditions are satisfied:
	\begin{thmenumerate}[label=(DS\arabic*), ref=(DS\arabic*), leftmargin=*]
		\item \label{i::stratification:stratumIsManifold} 
			Every piece \( X_\sigma \) is a locally closed manifold (whose manifold topology coincides with the relative topology).
			We will call \( X_\sigma \) a \emphDef{stratum} of \( X \). 
		\item \label{i::stratification:frontierCondition} (frontier condition)
			Every pair of disjoint strata \( X_\varsigma \) and \( X_\sigma \) with \( X_\varsigma \cap \closureSet{X_\sigma} \neq \emptyset \) satisfies:
			\begin{thmenumerate}[label=\alph*), ref=(DS2\alph*)]
				\item \label{i:stratification:frontierConditionBoundary}
					\( X_\varsigma \) is contained in the frontier \( \closureSet{X_\sigma} \setminus X_\sigma \) of \( X_\sigma \).
				\item \label{i:stratification:frontierConditionIntersection}
					\( X_\sigma \) does not intersect with \( \closureSet{X_\varsigma} \).
			\end{thmenumerate}
			In this case, we write \( X_\varsigma < X_\sigma \) or \( \varsigma < \sigma \) and call \( X_\varsigma \) a \emphDef{frontier stratum} of \( X_\sigma \). 
	\end{thmenumerate}		
	The pair \( (X, \stratification{Z}) \) is called a \emphDef{stratified space}.
\end{defn}

The relation \( X_\varsigma < X_\sigma \) introduced in \iref{i::stratification:frontierCondition} is indeed a strict partial order on the set of strata due to the property \iref{i:stratification:frontierConditionIntersection}, see \parencite[Proposition~2.5]{KondrackiRogulski1983} (which justifies the notation).

\begin{remarks}
	\item
		In some cases, one can use the manifold structure of the stratum combined with \iref{i:stratification:frontierConditionBoundary} to show that \iref{i:stratification:frontierConditionIntersection} is automatically satisfied.
		As \textcite[Theorem~3.1]{KondrackiRogulski1983} discuss, this happens, for example, if all strata are Banach manifolds.
		This observation also explains why \iref{i:stratification:frontierConditionIntersection} is rarely included in the definition of a stratification in the finite-dimensional setting.
	\item 
		The strata are often indexed by a countable set and we will call such a stratification \emphDef{countable}.
	\item
		In the finite-dimensional context, one usually assumes that the stratification is locally finite, \ie that every point has a neighborhood which intersects only finitely many strata (see, for example, \parencite[Section~4.1]{Sniatycki2013}).
		The orbit type stratification of a finite-dimensional proper \( G \)-action is locally finite as a result of the fact that an action of a compact group on a finite-dimensional vector space has only finitely many orbit types.
		However, as we have seen in \cref{sec:orbitTypeStratification:locallyFinite}, this is no longer the case for infinite-dimensional representations.
		Thus, in general, demanding a locally finite stratification is a too strong requirement in infinite dimensions.
	\item 
		We note that the terminology employed in the theory of stratified spaces varies from author to author.
		For example in \parencite[Chapter~1]{Pflaum2001}, \parencite[Chapter~1]{Pflaum2001a} and \parencite[Chapter~1]{OrtegaRatiu2003} the notion of a decomposed space is used for what we call a stratification.
		A stratification (in their sense) assigns to every point a set germ in such a way that (locally) a decomposition is induced.
		In the finite-dimensional setting, one can pass from every such stratification to a canonical decomposition which satisfies a certain minimality condition, see \parencite[Proposition~1.2.7]{Pflaum2001a} or \parencite[Proposition~1.11]{Pflaum2001}.
		Our terminology is more in line with the common usage in infinite dimensions, compare \parencites{KondrackiRogulski1983}[p.~551]{AbbatiCirelliEtAl1986}[Definition~4.4.4]{KondrackiRogulski1986}.

		Finally, in our context, it is convenient to use the slightly more general notion of a manifold as proposed in \parencites[Section~II.1]{Lang1999}[Definition~3.1.1]{MarsdenRatiuEtAl2002}.
		There, different connected components of the manifold are allowed to be modeled on non-isomorphic vector spaces.
		The related concept of a \( \Sigma \)-manifold is discussed in \parencite[Remark~1.1.4]{Pflaum2001a}.
\end{remarks}

The frontier condition allows us to write the closure of a stratum as the union of smaller strata.
\begin{lemma} \label{prop::stratification:frontierConditionImpliesClosureFormula}
	If \( X_\sigma \) is a stratum in the stratified space \( X \),  then
	\begin{equation}
		\closureSet{X_\sigma} = \bigUnion_{\varsigma \leq \sigma} X_\varsigma. 
		\qedhere
	\end{equation}
\end{lemma}
\begin{proof}
	The frontier condition \iref{i::stratification:frontierCondition} implies that every stratum \( X_\varsigma \) with \( X_\varsigma \cap \closureSet{X_\sigma} \neq \emptyset \) is completely contained in \( \closureSet{X_\sigma} \).
	Thus, \( \bigUnion_{\varsigma \leq \sigma} X_\varsigma \subseteq \closureSet{X_\sigma} \).
	For the converse inclusion, observe that every point in the closure of \( X_\sigma \) has to lie in some stratum \( X_\varsigma \).
	Thus, \( X_\varsigma \cap \closureSet{X_\sigma} \neq \emptyset \) and hence by definition \( \varsigma \leq \sigma \).
\end{proof}
The converse of the previous observation is true as well.
\begin{prop} \label{prop::stratification:ClosureFormulaImpliesFrontierCondition}
	Let \( X \) be a Hausdorff topological space partitioned into locally closed manifolds \( X_\sigma \) with \( \sigma \in \Sigma \).
	Assume that a partial order \( \unlhd \) is defined on the index set \( \Sigma \).
	If the closure of every piece \( X_\sigma \) can be written as
	\begin{equation} \label{eq::stratification:closureOfStrata}
		\closureSet{X_\sigma} = \bigUnion_{\varsigma \unlhd \sigma} X_\varsigma \, ,
	\end{equation}
	then the partition satisfies the frontier condition \iref{i::stratification:frontierCondition} and hence it is a stratification.
	Moreover, the partial order \( \leq \) defined by the frontier condition coincides with \( \unlhd \).
\end{prop}
\begin{proof}
	Let \( X_\sigma \) and \( X_\varsigma \) be two disjoint pieces such that \( X_\sigma \cap \closureSet{X_\varsigma} \neq \emptyset \).
	Then,~\eqref{eq::stratification:closureOfStrata} implies that \( X_\varsigma \) coincides with one of the pieces in the union \(  \bigUnion_{\varsigma \unlhd \sigma} X_\varsigma \).
	Thus, \( \varsigma \unlhd \sigma \) and the first part \iref{i:stratification:frontierConditionBoundary} of the frontier condition is satisfied.
	\sideRemark{
		The proof shows that \iref{i:stratification:frontierConditionBoundary} is equivalent to \( \closureSet{X_\sigma} = \bigUnion_{\varsigma \leq \sigma} X_\varsigma \).
		The purpose of \iref{i:stratification:frontierConditionIntersection} is to ensure that the incidence relation \( \varsigma \leq \sigma \) is a partial order.
	}

	It remains to verify \iref{i:stratification:frontierConditionIntersection}.
	Suppose, for the sake of contradiction, that the intersection \( \closureSet{X_\varsigma} \intersect X_\sigma \) is non-empty.
	By~\eqref{eq::stratification:closureOfStrata}, this is only possible if \( \sigma \unlhd \varsigma \).
	Hence, in total, \( \sigma \unlhd \varsigma \unlhd \sigma \).
	Since \( \unlhd \) is a partial order, we conclude \( \sigma = \varsigma \) which leads to a contradiction as the pieces \( X_\sigma \) and \( X_\varsigma \) were assumed to be disjoint.
\end{proof}

\begin{refcontext}[sorting=nyt]{}
	\printbibliography
\end{refcontext}

\end{document}